\documentclass[10pt,twoside,a4paper]{article}

\def\ShowComments{}

\usepackage[utf8]{inputenc}

\usepackage[T1]{fontenc}

\usepackage{amsmath}
\numberwithin{equation}{section}

\usepackage{amssymb}

\usepackage{microtype}

\usepackage[backend=bibtex,maxnames=10,backref=true,hyperref=true,safeinputenc]{biblatex}
\DefineBibliographyStrings{english}{%
	backrefpage = {cited on page},
	backrefpages = {cited on pages},
}

\DeclareFieldFormat[report]{title}{``#1''}
\DeclareFieldFormat[book]{title}{``#1''}
\AtEveryBibitem{\clearfield{url}}
\AtEveryBibitem{\clearfield{note}}

\usepackage{graphicx}
\graphicspath{{images/}}

\usepackage{tikz}
\usepackage{pgfplots}
\pgfplotsset{compat=newest}

\title{Regularization with Metric Double Integrals of Functions with Values in a Set of Vectors}

\author{René Ciak$^1$\\{\footnotesize\href{mailto:rene.ciak@univie.ac.at}{rene.ciak@univie.ac.at}}
	\and Melanie Melching$^1$\\{\footnotesize\href{mailto:melanie.hirzmann@univie.ac.at}{melanie.melching@univie.ac.at}}
	\and Otmar Scherzer$^{1,2}$\\{\footnotesize\href{mailto:otmar.scherzer@univie.ac.at}{otmar.scherzer@univie.ac.at}}}
\date{December 21, 2018}

\ifdefined\ShowTodos{}  
\usepackage{todonotes}  
\else  
\usepackage[disable]{todonotes}  
\fi

\newcommand{\todolit}[1][]{\ifthenelse{\isempty{#1}}{\todo[color=blue!20]{Citation}}{\todo[color=blue!20, inline]{#1}}}

\newenvironment{thm}{\begin{theorem}}{\end{theorem}}

\usepackage[english]{babel}
\usepackage{xparse}
\usepackage{etoolbox}     
\usepackage{mathtools}
\usepackage{stmaryrd}  
\usepackage{enumerate}  
\usepackage{enumitem}
\usepackage{bbm}
\usepackage[toc,page]{appendix}
\usepackage{float}
\usepackage{csquotes}
\usepackage[section]{placeins}
\usepackage{xifthen}
\usepackage{tikz}
\pgfplotsset{compat=newest}
\usepackage{algorithm2e}
\usepackage{pict2e}
\usepackage{algpseudocode}
\usepackage{graphicx}
\usepackage[labelformat=simple]{subcaption}

\usepackage{prettyref}
\newrefformat{lem}{Lemma \ref{#1}}
\newrefformat{thm}{Theorem \ref{#1}}
\newrefformat{cor}{Corollary \ref{#1}}
\newrefformat{rem}{Remark \ref{#1}}
\newrefformat{det}{Detail \ref{#1}}
\newrefformat{enu}{part \textup{\ref{#1})}}
\newrefformat{fig}{Figure \ref{#1}}
\usepackage{booktabs}

\newcommand{\sphere}{\mathbb{S}^1} 
\newcommand{\stripe}{\mathcal{S}}
\newcommand{\stripet}{\mathcal{S}_{\tau}}
\newcommand{\Sc}{\mathcal{S}^c}

\NewDocumentCommand \Lp  { d<> d[] d[] }     { \IfNoValueTF{#3}
	{ \IfNoValueTF{#1} {L^p(\Omega, \IfNoValueTF{#2}{\R^M}{#2})} {L^{p_{#1}}(\Omega_{#1}, K_{#1})}  }
	{ \IfNoValueTF{#1} {L^p(\Omega, #3)} {L^{p_{#1}}(\Omega_{#1}, #3)}  }
}
\NewDocumentCommand \Wsp { d<> d[] d[] }     { \IfNoValueTF{#3}    
	{\IfNoValueTF{#1} {W^{s,p}(\Omega, \IfNoValueTF{#2}{\R^M}{#2})} {W^{s,p_{#1}}(\Omega_{#1}, K_{#1})}}
	{ \IfNoValueTF{#1} {W^{s,p}(\Omega, #3)} {W^{s,p_{#1}}(\Omega_{#1}, #3)}  } 
}
\NewDocumentCommand \W   { d<> d[] }     {  \IfNoValueTF{#1} {W(\Omega, \IfNoValueTF{#2}{\R^M}{#2})} {W(\Omega_{#1}, K_{#1})}  }
\NewDocumentCommand \BV  { d<> d[] }     {  \IfNoValueTF{#1} {BV(\Omega, \IfNoValueTF{#2}{\R^M}{#2})} {BV(\Omega_{#1}, K_{#1})}  }

\newcommand{\Reg}{\mathcal{R}}
\newcommand{\normN}[1][]{\ifthenelse{\isempty{#1}}{\|\cdot\|_{\R^N}}{\|#1\|_{\R^N}}}
\newcommand{\normM}[1][]{\ifthenelse{\isempty{#1}}{\|\cdot\|_{\R^M}}{\|#1\|_{\R^M}}}
\newcommand{\normMo}[1][]{\ifthenelse{\isempty{#1}}{\|\cdot\|_{\R^{M_1}}}{\|#1\|_{\R^{M_1}}}}

\NewDocumentCommand \metric   { d[] m }   { \mathrm{d}_{#2} \IfValueT{#1}{|_{#1 \times #1}} }
\NewDocumentCommand \dist     { d[] }     {  \metric[#1]{}          }
\NewDocumentCommand \dRM      { d[] }     {  \metric[#1]{\R^M}      }
\NewDocumentCommand \dRMi     { d[] }     {  \metric[#1]{\R^{M_i}}  }
\NewDocumentCommand \dRMo     { d[] }     {  \metric[#1]{\R^{M_1}}  }
\NewDocumentCommand \dRMt     { d[] }     {  \metric[#1]{\R^{M_2}}  }
\NewDocumentCommand \dK       { d[] }     {  \metric[#1]{K}         }
\NewDocumentCommand \dKo      { d[] }     {  \metric[#1]{1}         }
\NewDocumentCommand \dKt      { d[] }     {  \metric[#1]{2}         }
\NewDocumentCommand \dS       { d[] }     {  \metric[#1]{\sphere}   }

\NewDocumentCommand \mebr     { O{\cdot} O{\cdot}   }          { \llbracket #1, #2 \rrbracket }
\NewDocumentCommand \F    { s d<> d<> d<> d<> }  
{
	\IfBooleanTF #1
	{  \mathcal{F} \IfValueTF{#3}{ [#2,#3] }{ \IfValueT{#2}{[#2{\dist}_1,#2{\dist}_2]} }  }
	{  \mathcal{F} \IfValueT{#2}{^{#2}} \IfValueT{#3}{_{#3}}  \IfValueTF{#5}{ [#4,#5] }{ \IfValueT{#4}{[#4{\dist}_1,#4{\dist}_2]} }  }  
}
\NewDocumentCommand \FT   { d<> d<> }  {  \tilde{\mathcal{F}} \IfValueT{#1}{^{#1}} \IfValueT{#2}{_{#2}} }

\DeclareMathOperator{\argmin}{argmin}

\newcommand{\op}[1][]{\ifthenelse{\isempty{#1}}{\mathrm{F}}{\mathrm{F}( #1)}}

\newcommand{\dx}{\,\mathrm{d}x}
\newcommand{\dy}{\,\mathrm{d}y}

\newcommand{\dxy}{\,\mathrm{d}(x,y)}

\NewDocumentCommand \rarr  {d<>}   { \IfNoValueTF{#1}{\rightarrow}{\overset{{#1}}{\rightarrow}} }
\newcommand{\Wto}{\rarr<W>}
\newcommand{\defeq}{\vcentcolon=}
\newcommand{\eqdef}{=\vcentcolon}

\newcommand{\pinfty}{+\infty}
\newcommand{\etau}{\eta_{\tau}}

\NewDocumentCommand \seq      { m O{n} o D<>{\mathbb{N}} }     { (#1_{#2})_{\IfNoValueTF{#3}{#2}{#3} \in #4} }
\NewDocumentCommand \subseq   { m O{n} o D<>{\mathbb{N}} }     { (#1_{#2'})_{\IfNoValueTF{#3}{#2'}{#3'} \in #4'} }
\NewDocumentCommand \lime     { O{n} }                         {    \lim_{#1 \rightarrow \infty} }
\NewDocumentCommand \limi     { O{n} }                         { \liminf_{#1 \rightarrow \infty} }
\NewDocumentCommand \lims     { O{n} }                         { \limsup_{#1 \rightarrow \infty} }

\usepackage[pdftex,colorlinks=true,linkcolor=blue,citecolor=green,urlcolor=blue,bookmarks=true,bookmarksnumbered=true]{hyperref}
\hypersetup
{
	pdfauthor={Ciak, Melching, Scherzer},
	pdfsubject={Subject},
	pdftitle={Regularization with Metric Double Integrals of Functions with Values in a Set of Vectors},
	pdfkeywords={regularization, manifold-valued data, non-convex, metric, double integral, fractional Sobolev space, bounded variation}
}

\usepackage[hyperref,amsmath,thmmarks]{ntheorem}
\usepackage{aliascnt}

\newtheorem{lemma}{Lemma}[section]

\newaliascnt{proposition}{lemma}
\newtheorem{proposition}[proposition]{Proposition}
\aliascntresetthe{proposition}

\newaliascnt{example}{lemma}
\newtheorem{example}[example]{Example}
\aliascntresetthe{example}

\newaliascnt{remark}{lemma}
\newtheorem{remark}[remark]{Remark}
\aliascntresetthe{remark}

\newaliascnt{notation}{lemma}

\aliascntresetthe{notation}

\newaliascnt{corollary}{lemma}
\newtheorem{corollary}[corollary]{Corollary}
\aliascntresetthe{corollary}

\newaliascnt{theorem}{lemma}
\newtheorem{theorem}[theorem]{Theorem}
\aliascntresetthe{theorem}

\theorembodyfont{\normalfont}
\newaliascnt{definition}{lemma}
\newtheorem{definition}[definition]{Definition}
\aliascntresetthe{definition}

\newaliascnt{assumption}{lemma}
\newtheorem{assumption}[assumption]{Assumption}
\aliascntresetthe{assumption}

\theoremstyle{nonumberplain}
\theoremseparator{:}
\theoremheaderfont{\normalfont\itshape}

\theoremsymbol{\ensuremath{\square}}
\newtheorem{proof}{Proof}

\usepackage[a4paper,centering,bindingoffset=0cm,marginpar=2cm,margin=2.5cm]{geometry}

\usepackage[pagestyles]{titlesec}
\titleformat{\section}[block]{\large\sc\filcenter}{\thesection.}{0.5ex}{}[]
\titleformat{\subsection}[runin]{\bf}{\thesubsection.}{0.5ex}{}[.]

\newpagestyle{headers}%
{%
	\headrule
	\sethead%
	[\footnotesize\thepage]%
	[\footnotesize\sc Ciak, Melching, Scherzer]%
	[]%
	{}%
	{\footnotesize\sc Regularization of Functions with Values in a Set}%
	{\footnotesize\thepage}
	\setfoot{}{}{}
}
\usepackage[font=footnotesize,format=plain,labelfont=sc,textfont=sl,width=0.75\textwidth,labelsep=period]{caption}
\usepackage{dsfont}
\newcommand{\N}{\mathds{N}}

\newcommand{\R}{\mathds{R}}

\let\RE\Re
\let\Re=\undefined
\DeclareMathOperator{\Re}{\RE e}
\let\IM\Im
\let\Im=\undefined
\DeclareMathOperator{\Im}{\IM m}

\newcommand{\abs}[1]{\left|#1\right|}
\newcommand{\norm}[1]{\left\|#1\right\|}
\newcommand{\set}[1]{\left\{#1\right\}}

\newcommand{\e}{\mathrm e}
\let\ii\i
\renewcommand{\i}{\mathrm i}
\renewcommand{\d}{\,\mathrm d}

\ifdefined\ShowComments{}
  \newcommand{\commentO}[1]{\textcolor{red}{#1}}
  \newcommand{\commentM}[1]{\textcolor{magenta}{M: #1}}
  \newcommand{\commentR}[1]{\textcolor{olive}{R: #1}}
\else
  \newcommand{\commentO}[1]{}
  \newcommand{\commentM}[1]{}
  \newcommand{\commentR}[1]{}
\fi

\newcommand{\ve}{\varepsilon}

\begin{document}
\renewcommand{\sectionautorefname}{Section}
\renewcommand{\subsectionautorefname}{Subsection}
\maketitle
\thispagestyle{empty}
\begin{center}
\hspace*{5em}
\parbox[t]{12em}{\footnotesize
\hspace*{-1ex}$^1$Computational Science Center\\
University of Vienna\\
Oskar-Morgenstern-Platz 1\\
A-1090 Vienna, Austria}
\hfil
\parbox[t]{17em}{\footnotesize
\hspace*{-1ex}$^2$Johann Radon Institute for Computational\\
\hspace*{1em}and Applied Mathematics (RICAM)\\
Altenbergerstraße 69\\
A-4040 Linz, Austria}
\end{center}

  \ifdefined\ShowTableOfContents{}  \tableofcontents  \fi

  \begin{abstract}
	We present an approach for variational regularization of inverse and imaging problems for recovering 
	functions with values in a set of vectors.
	We introduce regularization functionals, which are derivative-free double integrals of such functions. These 
	regularization functionals are motivated from double integrals, which approximate Sobolev semi-norms of 
	intensity functions. These were introduced in  
	Bourgain, Brézis \& Mironescu, \emph{``Another Look at Sobolev Spaces''.} In: \emph{Optimal Control and
        Partial Differential Equations}-Innovations \& Applications, IOS press, Amsterdam, 2001.
	For the proposed regularization functionals we prove 
	existence of minimizers as well as a stability and convergence result for functions with values in a set of vectors.
\end{abstract}

\section{Introduction}
Functions with values in a (nonlinear) subset of a vector space appear in several applications of imaging and in inverse problems, e.g.

\begin{itemize}
 \item \emph{Interferometric Synthetic Aperture Radar (InSAR)} 
       is a technique used in remote sensing and geodesy to generate for example digital elevation 
       maps of the earth's surface. 
       InSAR images represent phase differences of waves between two or more SAR images, cf. \cite{LiuMas16,RocPraFer97}. 
       Therefore InSAR data are functions $f:\Omega \to \sphere \subseteq \R^2$.
       The pointwise function values are on the $\sphere$, which is considered embedded into $\R^2$.
 \item A \emph{color image} can be represented as a function in \emph{HSV}-space (hue, saturation, value) (see e.g. \cite{PlaVen00}).
       Color images are then described as functions $f:\Omega \to K \subseteq \R^3$. Here $\Omega$ is a plane in $\R^2$, the image domain,
       and $K$ (representing the HSV-space) is a cone in 3-dimensional space $\R^3$. 
 \item Estimation of the \emph{foliage angle distribution} has been considered for instance in
       \cite{HelAndRobFin15,PutBriManWiePfeZliPfe16}.
       Thereby the imaging function is from $\Omega \subset \R^2$, a part of the Earth's surface, 
       into $\mathbb{S}^2 \subseteq \R^3$, representing foliage angle orientation.
 \item Estimation of functions with values in $SO(3) \subseteq \R^{3 \times 3}$. Such problems appear in \emph{Cryo-Electron Microscopy} 
       (see for instance \cite{HadSin11,SinShk12,WanSinWen13}).
\end{itemize}
We emphasize that we are analyzing \emph{vector}, \emph{matrix}, \emph{tensor}-valued functions, where pointwise function 
evaluations belong to some given (sub)set, but are always \emph{elements} of the underlying vector space. This should not be confused with set-valued functions, where every function 
evaluation can be a set.

Inverse problems and imaging tasks, such as the ones mentioned above, might be unstable, or even worse, the solution could be 
ambiguous. Therefore, numerical algorithms for imaging need to be \emph{regularizing} to obtain approximations of the 
desired solution in a stable manner. 
Consider the operator equation 
\begin{equation}\label{eq:basic_problem}
\op[w] = v^0,
\end{equation} 
where we assume that only (noisy) measurement data $v^\delta$ of $v^0$ become available. 
In this paper the method of choice is \emph{variational regularization} which consists 
in calculating a minimizer of the variational regularization functional
\begin{equation}
 \label{eq:energy}
 \mathcal{F}(w) \defeq \mathcal{D}(\op[w],v^\delta) + \alpha \mathcal{R}(w).
\end{equation}
Here
\begin{description}
 \item{$w$} is an element of the \emph{set} of admissible functions.
 \item{$\op$} is an operator modeling the image formation process (except the noise).
 \item{$\mathcal{D}$} is called the \textit{data} or \textit{fidelity term}, which is used to compare a pair 
       of data in the image domain, that is to quantify the difference of the two data sets. 
 \item{$\mathcal{R}$} is called \textit{regularization functional}, which is used to impose certain 
       properties onto a minimizer of the regularization functional $\mathcal{F}$.
 \item{$\alpha > 0$} is called \textit{regularization parameter} and provides a trade off between stability 
       and approximation properties of the minimizer of the regularization functional $\mathcal{F}$.  
 \item{$v^\delta$} denotes measurement data, which we consider noisy. 
 \item{$v^0$} denotes the exact data, which we assume to be not necessarily available.
\end{description}

The main objective of this paper is to introduce a general class of regularization functionals for functions with values in a set of vectors. 
In order to motivate our proposed class of regularization functionals we review a class of regularization functionals 
appropriate for analyzing \emph{intensity data}.

\subsection*{Variational regularization for reconstruction of intensity data}
Opposite to what we consider in the present paper, most commonly, imaging data $v$ and admissible functions $w$, respectively, are 
considered to be representable as \emph{intensity functions}. That is, they are functions from some subset $\Omega$ of an Euclidean 
space with \emph{real values}.

In such a situation the most widely used regularization functionals use regularization terms 
consisting of powers of Sobolev (see \cite{BouSau93,ChaLio97,CimMel12}) or total variation semi-norms \cite{RudOshFat92}. 
It is common to speak about \emph{Tikhonov regularization} (see for instance \cite{TikArs77}) when the data term and the 
regularization functional are squared Hilbert space norms, respectively. 
For the \emph{Rudin, Osher, Fatemi (ROF)} regularization \cite{RudOshFat92}, also known as total variation 
regularization, the data term is the squared $L^2$-norm and $\mathcal{R}(w) = |w|_{TV}$ is the total variation semi-norm. 
Nonlocal regularization operators based on the generalized nonlocal gradient is used in \cite{GilOsh08}. \\  
Other widely used regularization functionals are \emph{sparsity promoting} \cite{DauDefDem04,KolLasNiiSil12},
\emph{Besov space norms} \cite{LorTre08,LasSak09} and anisotropic regularization norms \cite{OshEse04,SchWei00}.
Aside from various regularization terms there also have been proposed different fidelity terms other than quadratic norm
fidelities, like the $p$-th powers of $\ell^p$ and $L^p$-norms of the differences of 
$F(w)$ and $v$ , \cite{SchGraGroHalLen09,SchuKalHofKaz12}, Maximum Entropy \cite{Egg93,EngLan93} and Kullback-Leibler divergence 
\cite{ResAnd07} (see \cite{Poe08} for some reference work).

Our work utilizes results from the seminal paper of \citeauthor{BouBreMir01} \cite{BouBreMir01}, which 
provides an equivalent \emph{derivative-free} characterization of Sobolev spaces and the space $\BV$, 
the space of functions of bounded total variation, which consequently, in this context, was analyzed in \citeauthor{Dav02} 
and \citeauthor{Pon04b} \cite{Dav02,Pon04b}, respectively. It is shown in 
\cite[Theorems 2 \& 3']{BouBreMir01} and \cite[Theorem 1]{Dav02} that when $(\rho_\ve)_{\ve > 0}$ 
is a suitable sequence of non-negative, radially symmetric, radially decreasing mollifiers, then
\begin{equation}
\label{eq:double_integral}
\begin{aligned}
\lim_{\ve \searrow 0} \tilde{\Reg}_\ve(w) & 
\defeq \lim_{\ve \searrow 0}  \int\limits_{\Omega\times \Omega} \frac{\|w(x)- w(y)\|_{\R}^p}{\normN[x-y]^{p}} \rho_\ve(x-y) \dxy\\
&= 
\begin{cases}
C_{p,N}|w|^p_{W^{1,p}} & \mbox{if } w \in W^{1,p}(\Omega, \R), \ 1 < p < \infty, \\
C_{1,N}|w|_{TV} & \mbox{if } w \in \BV[\R], \ p = 1, \\
\infty & \mbox{otherwise},
\end{cases}
\end{aligned}
\end{equation}
Hence $\tilde{\Reg}_\ve$ approximates powers of Sobolev semi-norms and the total variation semi-norm, respectively. Variational 
imaging, consisting in minimization of $\mathcal{F}$ from \autoref{eq:energy} with $\Reg$ replaced by 
$\tilde{\Reg}_\ve$,  has been considered in \cite{AubKor09,BouElbPonSch11}.

\subsection*{Regularization of functions with values in a set of vectors}
In this paper we generalize the derivative-free characterization of Sobolev spaces and functions of bounded variation to 
functions, $u:\Omega \to K$, where $K$ is some set of vectors, and use these functionals for variational 
regularization. The applications we have in mind contain that $K$ is a closed subset of $\R^M$ (for instance HSV-data) 
with non-zero measure, or that $K$ is a sub-manifold (such as for instance InSAR-data).

The reconstruction of manifold--valued data with variational regularization methods has already been subject 
to intensive research (see for instance \cite{KimSoc02,CreStr11,CreStr13,CreKoeLelStr13,BacBerSteWei16,WeiDemSto14}).
The variational approaches mentioned above use regularization and fidelity functionals based on 
Sobolev and TV semi-norms: a total variation regularizer 
for cyclic data on $\sphere$ was introduced in \cite{CreStr13,CreStr11}, see also \cite{BerLauSteWei14,BerWei16,BerWei15}.
In \cite{BacBerSteWei16,BerFitPerSte17} combined first and second order 
differences and derivatives were used for regularization to restore manifold--valued data. The later mentioned papers, however, 
are formulated in a finite dimensional setting, opposed to ours, which is considered in an infinite dimensional setting. 
Algorithms for total variation minimization problems, including half-quadratic minimization and non-local patch based methods, 
are given for example in \cite{BacBerSteWei16,BerChaHiePerSte16,BerPerSte16} as well as in \cite{GroSpr14,LauNikPerSte17}.
On the theoretical side the total variation of functions with values in a manifold was investigated by \citeauthor{GiaMuc07} 
using the theory of Cartesian currents in \cite{GiaMuc07,GiaMuc06}, and earlier \cite{GiaModSou93} if the manifold is a 
$\sphere$.

\subsection*{The contents and the particular achievements of the paper are as follows}
The contribution of this paper is to introduce and analytically analyze double integral 
regularization functionals for reconstructing functions with values in a set of vectors, generalizing functionals 
of the form \autoref{eq:double_integral}. 
Moreover, we develop and analyze fidelity terms for comparing manifold--valued data. 
Summing these two terms provides a new class of regularization functionals of the form 
\autoref{eq:energy} for reconstructing manifold--valued data. 

When analyzing our functionals we encounter several differences to existing 
regularization theory (compare \autoref{sec: Setting}):
\begin{enumerate}
\item 
      The \emph{admissible functions}, where we minimize the regularization functional on, 
      do form only a \emph{set} but \emph{not} a \emph{linear} space.
      As a consequence, well--posedness 
      of the variational method (that is, existence of a minimizer of the energy functional) cannot directly be proven 
      by applying standard 
      direct methods in the Calculus of Variations \cite{Dac82,Dac89}. 
\item The regularization functionals are defined via metrics and not norms, see \autoref{sec: Existence}. 
\item In general, the fidelity terms are \emph{non-convex}.
      Stability and convergence results are proven in \autoref{sec: Stability_and_Convergence}. 
\end{enumerate}      

The model is validated in \autoref{sec:Numerical_results} where we present numerical results for denoising 
and inpainting of data of InSAR type.

\section{Setting} \label{sec: Setting}
In the following we introduce the basic notation and the set of admissible functions which we are regularizing 
on.

\begin{assumption}
 \label{ass:1}
 All along this paper we assume that
 \begin{itemize}
  \item $p_1, p_2 \in [1, \pinfty)$, $s \in (0,1]$,
  \item $\Omega_1, \Omega_2 \subseteq \R^N$ are nonempty, bounded, and connected open sets with Lipschitz boundary, respectively,
  \item $k \in [0,N]$,
  \item $K_1 \subseteq \R^{M_1}, K_2 \subseteq \R^{M_2}$ are nonempty and closed subsets of $\R^{M_1}$ and $\R^{M_2}$, respectively.       
 \end{itemize}
 Moreover, 
 \begin{itemize}
  \item $\normN$ and $\|\cdot\|_{\R^{M_i}}, \ i=1,2,$ are the Euclidean norms on $\R^N$ and $\R^{M_i}$, respectively.
  \item $\dRMi: \R^{M_i} \times \R^{M_i} \rarr [0, \pinfty)$ denotes the Euclidean distance on $\R^{M_i}$ for $i=1,2$ and 
  \item $\d_i \defeq \mathrm{d}_{K_i}: K_i \times K_i \rarr [0, \pinfty)$ 
          denote arbitrary metrics on $K_i$, which fulfill for $i=1$ and $i=2$
          \begin{itemize}
          	\item $\dRMi|_{K_i \times K_i} \leq d_i$,
          	\item $\d_i$ is continuous with respect to $\dRMi|_{K_i \times K_i}$, meaning that for a sequence $\seq{a}$ in $K_i \subseteq \R^{M_i}$  converging to some $a \in K_i$ we also have $\d_i(a_n,a) \rarr 0$. 
          \end{itemize} 
          In particular, this assumption is valid if the metric $d_i$ is equivalent to $\dRMi|_{K_i \times K_i}$.
          When the set $K_i, \ i=1,2$, is a suitable complete submanifold of $\R^{M_i}$,
          it seems natural to choose $d_i$ as the geodesic distance on the respective submanifolds.
  \item  $(\rho_{\ve})_{\ve > 0}$ is a Dirac family of non-negative, radially symmetric mollifiers, i.e. for every $\ve > 0$ we have
    \begin{enumerate}
      \item $\rho_\ve \in \mathcal{C}^{\infty}_{c}(\R^N, \R)$ is radially symmetric,
      \item $\rho_\ve \geq 0$,
      \item $\int \limits_{\R^N} \rho_\ve (x) \dx = 1$, and 
      \item for all $\delta > 0$, $\lim_{\ve \searrow 0}\limits \int_{\set{\normN[y] > \delta}} \rho_\ve(y) \dy = 0$.
    \end{enumerate} 
    We demand further that, for every $\ve > 0$, 
    \begin{enumerate}[resume]
      \item there exists a  $\tau > 0$ and $\etau > 0$ such that
        $\{ z \in \R^N : \rho_{\ve}(z) \geq \tau \}= \{z \in \R^N : \normN[z] \leq \etau \}$.
    \end{enumerate}
    This condition holds, e.g., if $\rho_{\ve}$ is a radially decreasing continuous function with $\rho_{\ve}(0) > 0$.
  \item When we write $p$, $\Omega$, $K$, $M$, then we mean $p_i$, $\Omega_i$, $K_i$, $M_i$, for either 
        $i=1,2$. In the following we will often omit the subscript indices whenever possible.
  \end{itemize}
\end{assumption}

\begin{example}\label{ex:mol}
Let $\hat{\rho} \in C_c^\infty(\R,\R_+)$ be symmetric at $0$, monotonically decreasing on $[0, \infty)$ and satisfy
\begin{equation*}
 \abs{\mathbb{S}^{N-1}}\int_0^\infty \hat{t}^{N-1} \hat{\rho}\left(\hat{t}\right)d \hat{t} = 1\;.
\end{equation*}
Defining mappings $\rho_\ve: \R^N \to \R$ by
\begin{equation*}
  \rho_\ve(x)\defeq \frac{1}{\ve^N} \hat{\rho}\left(\frac{\normN[x]}{\ve}\right)
\end{equation*}
constitutes then a family $(\rho_\ve)_{\ve > 0}$ which fulfills the above properties \it{(i) -- (v)}.
Note here that
\begin{itemize}
 \item by substitution $x = t \theta$ with $t > 0, \theta \in \mathbb{S}^{N-1}$ and $\hat{t}=\frac{t}{\ve}$, 
\begin{equation}
 \label{eq:molII}
 \begin{aligned}
  \int_{\R^N} \rho_\ve(x)\, \dx &= \frac{1}{\ve^N} \int_{\R^N} \hat{\rho}\left(\frac{\normN[x]}{\ve}\right) \dx \\
  &= \frac{1}{\ve^N} \int_0^\infty t^{N-1} \hat{\rho}\left(\frac{t}{\ve}\right)\d t \int_{\mathbb{S}^{N-1}} \d\theta  \\ 
  &= \abs{\mathbb{S}^{N-1}}\int_0^\infty \hat{t}^{N-1} \hat{\rho}\left(\hat{t}\right)\d \hat{t} = 1\;.
 \end{aligned}
\end{equation}
Here, $d\theta$ refers to the canonical spherical measure.
 \item Again by the same substitutions, taking into account that $\hat{\rho}$ has compact support, it follows for 
       $\ve > 0$ sufficiently small that
\begin{equation}
 \label{eq:molIIa}
 \begin{aligned}
  \int_{\set{y:\normN[y]>\delta}} \rho_\ve(x)\, \dx 
  &= \frac{1}{\ve^N} \int_{\set{y:\normN[y]> \delta}} \hat{\rho}\left(\frac{\normN[x]}{\ve}\right) \dx \\
  &= \frac{1}{\ve^N} \int_\delta^\infty t^{N-1} \hat{\rho}\left(\frac{t}{\ve}\right)\d t \int_{\mathbb{S}^{N-1}} \d\theta  \\ 
  &= \abs{\mathbb{S}^{N-1}}\int_{\delta/\ve}^\infty \hat{t}^{N-1} \hat{\rho}\left(\hat{t}\right)\d \hat{t} =0 \;.
 \end{aligned}
\end{equation}
\end{itemize}
\end{example}
In the following we write down the basic spaces and sets, which will be used in the course of the paper.
\begin{definition}
\begin{itemize}
 \item The \emph{Lebesgue--Bochner space} of $\R^M$--valued functions on $\Omega$ consists of the set 
  \begin{equation*}
  \begin{aligned}
    \Lp \defeq \{ \phi : \Omega \to \R^M : {} & \phi \text{ is Lebesgue-Borel measurable and }  \\ 
                                           &\normM[\phi(\cdot)]^p: \Omega \to \R \text{ is Lebesgue--integrable on } \Omega \},
   \end{aligned}
  \end{equation*}
  which is associated with the norm $\|\cdot\|_{\Lp}$, given by
  \begin{gather*}
    \norm{\phi}_{\Lp} \defeq \Big( \int_{\Omega}\limits \normM[\phi(x)]^p \dx \Big)^{1/p} \; .
  \end{gather*}
 \item Let $0 < s < 1$. Then the \emph{fractional Sobolev space} of order $s$ can be defined (cf. \cite{Ada75}) as the set
\begin{gather*}
  \Wsp \defeq \left\{ w \in \Lp : \frac{\normM[w(x) - w(y)]}{\normN[x-y]^{\frac{N}{p}+s}} \in 
  L^p (\Omega \times \Omega, \R)  \right\} \\
  = \{w \in \Lp : \abs{w}_{\Wsp} < \infty \},
\end{gather*}
equipped with the norm
\begin{equation}\label{eq:sobolev_norm}
 \|\cdot\|_{\Wsp} \defeq \big(\|\cdot\|_{\Lp}^{p} + \abs{\cdot}_{\Wsp}^p \big)^{1/p},
\end{equation}
where $\abs{\cdot}_{\Wsp}$ is the semi-norm for $\Wsp$, given by
\begin{equation}\label{eq:sobolev_semi_norm}
  \abs{w}_{\Wsp} \defeq  \Big(\int\limits_{\Omega\times \Omega} \frac{\normM[w(x) - w(y)]^p}{\normN[x-y]^{N+ps}} \dxy \Big)^{1/p},
  \quad w \in \Wsp\;.
\end{equation}
\item 
For $s = 1$ the Sobolev space $W^{1,p}(\Omega, \R^M)$ consists of all weakly differentiable
functions in $L^1(\Omega,\R^M)$ for which 
\begin{equation*}
\norm{w}_{W^{1,p}(\Omega, \R^M)} 
  \defeq \Big( \norm{w}_{\Lp}^p +  \int_{\Omega}\limits \|\nabla w(x)\|^p_{\mathbb{R}^{M\times N}} \dx \Big)^{1/p} < \infty\;,
\end{equation*}
where $\nabla w$ is the weak Jacobian of $w$.
\item Moreover, we recall one possible definition of the space $\BV$ from \cite{AmbFusPal00}, which consists of all 
      Lebesgue--Borel measurable functions $w:\Omega \to \R^M$ for which 
\begin{gather*}
\norm{w}_{\BV} \defeq \norm{w}_{L^1(\Omega, \R^M)} + \abs{w}_{\BV} < \infty,
\end{gather*}
where 
\begin{equation*}
\begin{aligned}
~ & \abs{w}_{\BV} \\
\defeq {}& \sup \left\{ \int \limits_\Omega w(x) \cdot \mathrm{Div} \varphi(x) \dx : \ 
\varphi \in C_c^1(\Omega, \R^{M \times N}) ~ \mathrm{ such~that }  \norm{\varphi}_\infty \defeq \mathop{\mathrm{ess~sup}}_{x \in \Omega} \norm{\varphi(x)}_F \leq 1 \right\},
\end{aligned}
\end{equation*}
where $\norm{\varphi(x)}_F$ is the Frobenius-norm of the matrix $\varphi(x)$ and 
$\mathrm{Div}\varphi = (\mathrm{div} \varphi_1, \dots, \mathrm{div} \varphi_M)^\mathrm{T}$ denotes the row--wise formed divergence of $\varphi$.
\end{itemize}
\end{definition}

\begin{lemma}
\label{le:inclusion}
 Let $0 < s \leq 1$ and $p \in [1,\infty)$, then $\Wsp \hookrightarrow \Lp$ and the embedding is compact. Moreover, the embedding
 $\BV \hookrightarrow \Lp$ is compact for all 
   \begin{gather*}
    1 \leq p < 1^* \defeq
    \begin{cases}
      \pinfty &\mbox{if } N = 1 \\
      \frac{N}{N-1} &\mbox{otherwise }
    \end{cases} \,.
  \end{gather*} 
\end{lemma}
\begin{proof}
 The first result can be found in \cite{DemDem07} for $0 < s < 1$ and in \cite{Eva10} for $s = 1$. The second assertion 
 is stated in \cite{AmbFusPal00}.
\end{proof}

\begin{remark} 
\label{re:notes_basic}
Let \autoref{ass:1} hold.
  We recall some basic properties of weak convergence in $\Wsp$, $W^{1,p}(\Omega, \R^M)$ and weak* convergence in $\BV$ (see for instance \cite{Ada75,AmbFusPal00}) :
  \begin{itemize}
    \item Let $p > 1$, $s\in(0,1]$ and assume that $(w_n)_{n \in \N}$ is bounded in $\Wsp$. 
          Then there exists a subsequence $(w_{n_k})_{k \in \N}$ which converges weakly in $\Wsp$. 
    \item Assume that $(w_n)_{n \in \N}$ is bounded in $\BV$. Then there exists a subsequence $(w_{n_k})_{k \in \N}$ 
          which converges weakly* in $\BV$. 
  \end{itemize}
\end{remark}

Before introducing the regularization functional, which we investigate theoretically and numerically, we give 
the definition of some sets of (equivalence classes of) admissible functions.
\begin{definition} \label{def:spaces_etc}
\label{de:basic}
    For $0 < s \leq 1$, $p \geq 1$ and a nonempty closed subset $K \subseteq \R^M$ we define 
    \begin{equation}
    \label{eq:spacess_K}
    \begin{aligned}
      \Lp[K]  \defeq {} & \{\phi \in \Lp : \phi(x) \in K \text{ for a.e. } x \in \Omega\}; \\
      \Wsp[K] \defeq {} & \{w \in \Wsp: w(x) \in K \text{ for a.e. } x \in \Omega  \}, \\
      \BV[K]  \defeq {} & \{w \in \BV: w(x) \in K \text{ for a.e. } x \in \Omega  \}.
    \end{aligned}
    \end{equation}
    and equip each of these (in general nonlinear) sets with some subspace topology:
    \begin{itemize}
    \item $\Lp[K] \subseteq \Lp$ is associated with the strong $\Lp$-topology,
    \item $\Wsp[K] \subseteq \Wsp$ is associated with the weak $\Wsp$-topology, and 
    \item $\BV[K] \subseteq \BV$ is associated with the weak* $\BV$-topology.
    \end{itemize}

    Moreover, we define 
      \begin{equation} \label{eq:ChooseW}
      W(\Omega,K) \defeq 
        \begin{cases}
          \Wsp[K] & \text{ for } p \in (1, \infty) \text { and } s \in (0,1], \\
          \BV[K]  & \text{ for } p = 1             \text { and } s = 1\;.
        \end{cases}
      \end{equation}
      Consistently, $W(\Omega,K)$ 
      \begin{itemize}
       \item is associated with the weak $\Wsp$-topology in the case $p \in (1, \infty)$ and $s \in (0,1]$ and
       \item with the weak* $\BV$-topology when $p=1$ and $s=1$.
      \end{itemize}
      When we speak about 
      \begin{equation*}
      \text{ convergence on } W(\Omega,K) \text{ we write } \overset{W(\Omega, K)}{\longrightarrow} \text{ or simply} \Wto
      \end{equation*}
      and mean weak convergence on $W^{s,p}(\Omega,K)$ and weak* convergence on $\BV[K]$, respectively.
\end{definition}
\begin{remark} 
\label{re:notes_choose_w}
~\nopagebreak
\begin{itemize}[topsep=0pt]
\item In general $\Lp[K], \Wsp[K]$ and $\BV[K]$ are sets which do not form a linear space.
\item If $K = \sphere$, then $\Wsp[K] = \Wsp[\sphere]$ as occurred in \cite{BouBreMir00b}.
\item For an embedded manifold $K$ the dimension of the manifold is not necessarily identical with the space dimension of $\R^M$. 
      For instance if $K = \sphere \subseteq \R^2$, then the dimension of $\sphere$ is $1$ and $M=2$.
\end{itemize}
\end{remark}
The following lemma shows that $W(\Omega,K)$ is a sequentially closed subset of $\W$.
\begin{lemma}[Sequential closedness of $W(\Omega,K)$ and {$\Lp[K]$}] \label{lem:Wsp_weakly_seq_closed_etc} 
  ~\nopagebreak 
  \vspace{-0.5\baselineskip} 
  \begin{enumerate}[topsep=0pt]
    \item 
      Let $w_* \in \W$ and $(w_n)_{n\in \N}$ be a sequence in $\W[K] \subseteq \W$ with 
      $w_n \overset{W(\Omega, \R^M)}{\longrightarrow} w_*$ as $n \to \infty$. 
      Then $w_* \in \W[K]$ and $w_n \rarr w_*$ in $\Lp[K]$.
    \item 
        Let $v_* \in \Lp$ and $(v_n)_{n \in \N}$ be a sequence in $\Lp[K] \subseteq \Lp$ with $v_n \to v_*$ in $\Lp$ as $n \to \infty$.
        Then $v_* \in \Lp[K]$  and there is some subsequence $(v_{n_k})_{k \in \N}$ which converges to $v_*$ pointwise almost everywhere, i.e. $v_{n_k}(x) \to v_*(x)$
        as $k \to \infty$ for almost every $x \in \Omega$.
  \end{enumerate}
\end{lemma}
\begin{proof}
  For the proof of the second part, cf. \cite{Els02}, Chapter VI, Corollary 2.7
  and take into account the closedness of $K \subseteq \R^M$.
  The proof of the first part follows from standard convergence arguments in $\Wsp$, $\BV$ and $\Lp$, respectively, 
  using the embeddings from \autoref{le:inclusion}, an argument on subsequences and part two.
\end{proof}

\begin{remark}
  \autoref{le:inclusion} along with \autoref{lem:Wsp_weakly_seq_closed_etc} imply that $\W[K]$ is compactly embedded in $\Lp[K]$, 
  where these sets are equipped with the bornology inherited from $\W$ and the topology inherited from $\Lp$, respectively.
\end{remark}

In the following we postulate the assumptions on the operator $\op$ which will be used throughout the paper:
\begin{assumption}
\label{ass:2}
Let $\W<1>$ be as in \autoref{eq:ChooseW} and assume that $\op$ is an operator from $\W<1>$ to $\Lp<2>$. 
\end{assumption}

We continue with the definition of our regularization functionals:
\begin{definition} \label{def:functional}
Let \autoref{ass:1} and \autoref{ass:2} hold. Moreover, let $\ve > 0$ be fixed and let 
$\rho \defeq \rho_\ve$ be a mollifier.

The regularization functional  
    $\F<v><\alpha>[\dKt, \dKo] : \W<1> \rarr [0, \infty]$ is defined as follows
    \begin{equation}\label{eq: functional_with_some_metric}
    \boxed{
      \F<v><\alpha>[\dKt, \dKo] (w) \defeq \int\limits_{\Omega_2} \dKt^{p_2}(\op[w](x), v(x)) \dx + \alpha \int\limits_{\Omega_1\times \Omega_1} \frac{\dKo^{p_1}(w(x), w(y))}{\normN[x-y]^{k+p_1 s}} \rho^l(x-y) \dxy,}
        \end{equation}
    where 
    \begin{enumerate}
      \item  $v \in \Lp<2>$,
      \item  $s \in (0,1]$, 
      \item  $\alpha \in (0, \pinfty)$ is the regularization parameter, 
      \item  $l \in \set{0, 1}$ is an indicator and
      \item  \label{itm:k}
      $\begin{cases}
      k \leq N &\mbox{if } \W<1> = W^{s,p_1}(\Omega_1, K_1), \ 0<s<1, \\
      k=0 & \mbox{if } \W<1> = W^{1,p_1}(\Omega_1, K_1) \text{ or if } \W<1> = BV(\Omega_1, K_1), \text{ respectively.}
      \end{cases}$
    \end{enumerate}
    Setting
    \begin{equation}\label{eq:d2}
    \boxed{
     \mebr[\phi][\nu]_{[\dKt]} \defeq \left( \int\limits_{\Omega_2} \dKt^{p_2}(\phi(x),\nu(x)) \dx \right)^{\frac{1}{p_2}},}
    \end{equation}
    and 
    \begin{equation}\label{eq:d3}
    \boxed{
     \Reg_{[\dKo]}(w) \defeq \int\limits_{\Omega_1\times \Omega_1} \frac{\dKo^{p_1}(w(x), w(y))}{\normN[x-y]^{k+p_1 s}} \rho^l(x-y) \dxy,}
    \end{equation}
    \autoref{eq: functional_with_some_metric} can be expressed in compact form
    \begin{equation}
      \label{eq:functional}
      \boxed{
        \F<v><\alpha>[\dKt,\dKo](w) = \mebr[\op[w]][v]^{p_2}_{[\dKt]} + \alpha \Reg_{[\dKo]}(w).}
    \end{equation}
    For convenience we will often skip some of the super- or subscript, and use compact notations like e.g. 
    \begin{equation*}
     \F<v>, \F[\dKt, \dKo] \text{ or } \F(w) = \mebr[\op[w]][v]^{p_2} + \alpha \Reg(w).
    \end{equation*}
\end{definition}

\begin{remark}
~
  \begin{enumerate}[topsep=0pt]
    \item $l = \set{0,1}$ is an indicator which allows to consider approximations of Sobolev semi-norms and double integral 
        representations 
        of the type of \citeauthor{BouBreMir01} \cite{BouBreMir01} in a uniform manner. 
        \begin{itemize}
         \item when $k=0$, $s=1$, $l=1$ and when $d_1$ is the Euclidean distance, we get the double integrals of the 
         \citeauthor{BouBreMir01}-form \cite{BouBreMir01}. Compare with \autoref{eq:double_integral}.
         \item When $d_1$ is the Euclidean distance, $k=N$ and $l=0$, we get Sobolev semi-norms.
        \end{itemize}
        We expect a relation between the two classes of functionals for $l=0$ and $l=1$ as stated in \autoref{ss:conjecture}.
    \item 
      When $d_1$ is the Euclidean distance then the second term in \autoref{eq: functional_with_some_metric} is similar to the ones used in 
      \cite{AubKor09,BouElbPonSch11} and \cite{BouBreMir01, Pon04b, Dav02}. 

  \end{enumerate}
\end{remark}

In the following we state basic properties of $\mebr_{[\dKt]}$ and the functional 
$\F$.

\begin{proposition} \label{pr:ExprIsOp}
Let \autoref{ass:1} hold. 

\begin{enumerate}
\item Then the mapping $ \mebr_{[\dKt]} : \Lp<2> \times \Lp<2> \rarr [0, \pinfty]$ 
satisfies the metric axioms.

\item \label{itm: ExpIsOp} Let, in addition, \autoref{ass:2} hold, assume that $v \in \Lp<2>$ and that both metrics $d_i$, $i=1,2$, 
are equivalent to $\dRMi|_{K_i \times K_i}$, respectively.
Then the functional $\F<v><\alpha>[\dKt, \dKo]$ does not 
attain the value $+\infty$ on its domain $\W<1> \neq \emptyset$. 
\end{enumerate}
\end{proposition}

\begin{proof}
\begin{enumerate}
\item The axioms of non-negativity, identity of indiscernibles and symmetry are fulfilled by $\mebr_{[\dKt]}$ 
      since $\dKt$ is a metric. To prove the triangle inequality let $\phi,\xi,\nu \in L^{p_2}(\Omega_2, K_2)$. 
      In the main case $\mebr[\phi][\nu]_{[\dKt]}^{p_2} \in (0, \infty)$ Hölder's inequality yields
	\begin{align*}
	\mebr[\phi][\nu]_{[\dKt]}^{p_2} ={}&
	\int\limits_{\Omega_2} \dKt\big(\phi(x),\nu(x) \big)  \dKt^{p_2-1}\big( \phi(x),\nu(x) \big)  \dx \\ 
	\leq{}&  \int\limits_{\Omega_2} \dKt\big( \phi(x),\xi(x) \big) \dKt^{p_2-1}\big( \phi(x),\nu(x) \big) \dx  
	      + \int\limits_{\Omega_2} \dKt\big( \xi(x),\nu(x) \big) \dKt^{p_2-1} \big( \phi(x),\nu(x) \big) \dx \\
  \leq{}&  
  \left( \int\limits_{\Omega_2} \dKt^{p_2} \big(\phi(x),\xi(x) \big) \dx  \right)^{\frac{1}{p_2}} 
       \left( \int\limits_{\Omega_2} \dKt^{p_2}\big( \phi(x),\nu(x) \big) \dx  \right)^{\frac{p_2-1}{p_2}}\\
  &+ \left( \int\limits_{\Omega_2} \dKt^{p_2} \big(\xi(x),\nu(x) \big) \dx  \right)^{\frac{1}{p_2}} 
       \left( \int\limits_{\Omega_2} \dKt^{p_2}\big( \phi(x),\nu(x) \big) \dx  \right)^{\frac{p_2-1}{p_2}}\\
	={}& \left( \mebr[\phi][\xi]_{[\dKt]} + \mebr[\xi][\nu]_{[\dKt]} \right)
	\mebr[\phi][\nu]_{[\dKt]}^{p_2-1},
	\end{align*} 
  meaning 
  \begin{gather*}
    \mebr[\phi][\nu]_{[\dKt]} \leq \mebr[\phi][\xi]_{[\dKt]} + \mebr[\xi][\nu]_{[\dKt]}.
  \end{gather*}
  If $\mebr[\phi][\nu]_{[\dKt]} = 0$ the triangle inequality is trivially fulfilled. 
  
  In the remaining case $\mebr[\phi][\nu]_{[\dKt]} = \infty$ applying the estimate $(a+b)^p \leq 2^{p-1} (a^p + b^p)$,
  see e.g. \cite[Lemma 3.20]{SchGraGroHalLen09},
  to $a = \dKt(\phi(x), \xi(x)) \geq 0$ and $b = \dKt(\xi(x), \nu(x)) \geq 0$ yields 
  \begin{gather*}
    \mebr[\phi][\nu]_{[\dKt]}^{p_2} \leq 2^{p_2-1} \big( \mebr[\phi][\xi]_{[\dKt]}^{p_2} + \mebr[\xi][\nu]_{[\dKt]}^{p_2} \big),
  \end{gather*}
  implying the desired result.
\item We emphasize that $\W<1> \neq \emptyset$ because every constant function $w(\cdot) = a \in K_1$ belongs to $\Wsp<1>$ 
      for $p_1 \in (1, \infty)$ and $s \in (0,1]$ as well as to $\BV<1>$ for $p_1 = 1$ and $s = 1$.
      Assume now that the metrics $d_i$ are equivalent to $\dRMi|_{K_i \times K_i}$ for $i=1$ and $i=2$, respectively,
      so that we have an upper bound $d_i \leq C \dRMi|_{K_i \times K_i}$.
      We need to prove that $\F<v><\alpha>[\dKt, \dKo](w) < \infty$ for every $w \in \W<1>$.
      Due to $\mebr[\phi][\nu]^{p_2}_{[\dKt]} \leq C^{p_2} \norm{\phi - \nu}^{p_2}_{\Lp<2>[][\R^{M_2}]} < \infty$ 
      for all $\phi, \nu \in \Lp<2> \subseteq \Lp<2>[][\R^{M_2}]$ it is sufficient to show
      $\Reg_{[\dKo]}(w) < \pinfty$ for all $w \in \W<1>$.
\begin{itemize}
  \item For $\W<1> = \BV<1>$ this is guaranteed by \cite[Theorem 1.2]{Pon04b}.
  \item For $\W<1> = W^{1,p_1}(\Omega_1, K_1)$ by \cite[Theorem 1]{BouBreMir01}.
  \item For $\W<1> = \Wsp<1>$, $s \in (0,1)$,  we distinguish between two cases. \\
      If $\normN[x-y]< 1$ we have that $\frac{1}{\normN[x-y]^{k+p_1 s}} \leq \frac{1}{\normN[x-y]^{N+p_1 s}}$ for $k \leq N$ and hence
      \begin{gather*}\int\limits_{\begin{smallmatrix}
      	(x,y) \in \Omega_1 \times \Omega_1 \\ \normN[x-y]< 1
      	\end{smallmatrix}} \frac{\dKo^{p_1}(w(x), w(y))}{\normN[x-y]^{k+p_1 s}} \rho^l(x-y) \dxy 
        \leq C^{p_1} \norm{\rho^l}_{\infty} \abs{w}_{W^{s,p_1}(\Omega_1, \R^{M_1})}^{p_1} < \infty\;.
      \end{gather*}
      If $\normN[x-y]\geq 1$ we can estimate 
      \begin{gather*}\int\limits_{\begin{smallmatrix}
      	(x,y) \in \Omega_1 \times \Omega_1 \\ \normN[x-y]\geq 1
      	\end{smallmatrix}} \frac{\dKo^{p_1}(w(x), w(y))}{\normN[x-y]^{k+p_1 s}} \rho^l(x-y) \dxy 
        \leq C^{p_1} \norm{\rho^l}_{\infty} 2^{p_1} |\Omega_1| \norm{w}^{p_1}_{L^{p_1}(\Omega_1, \R^{M_1})} < \infty\;.
     \end{gather*}
     In summary adding yields $\Reg_{[\dKo]}(w) < \pinfty$.
     \end{itemize}
\end{enumerate}
\end{proof}

\section{Existence} \label{sec: Existence}

In order to prove existence of a minimizer of the functional $\F$ we apply the Direct Method in the Calculus of Variations
(see e.g. \cite{Dac82,Dac89}). 
To this end we verify continuity properties of $\mebr_{[\dKt]}$ and $\Reg_{[\dKo]}$, resp. $\F[\dKt, \dKo]$ and apply them along with
the sequential closedness of $\W<1>$, already proven in \autoref{lem:Wsp_weakly_seq_closed_etc}.

In this context we point out some setting assumptions and their consequences on $\F$, resp. $\mebr$ and $\Reg$ in the following remark. 
For simplicity we assume 
$p \defeq p_1 = p_2 \in (1, \infty)$, $\Omega \defeq \Omega_1 = \Omega_2$
and $(K, \dK) \defeq (K_1, \dKo) = (K_2, \dKt)$.
\begin{remark} \label{re:tricks}
~
	\begin{itemize}[topsep=0pt]
		\item
		The continuity of $\dK$ with respect to $\dRM|_{K \times K}$ guarantees lower semicontinuity of $\mebr_{[\dK]}$ and $\Reg_{[\dK]}$.
		\item
		The inequality $\dRM|_{K \times K} \leq \dK$ carries over to the inequalities 
		$\norm{\widetilde v - v}_{\Lp} \leq \mebr[\widetilde v][v]_{[\dK]}$ for all $\widetilde v, v \in \Lp[K]$,
		and $|w|_{\W} \leq \Reg_{[\dK]}(w)$ for all $w \in \W[K]$, allowing to transfer
		properties like coercivity from $\F[\dRM,\dRM]$ to $\F[\dK,\dK]$.
		Moreover, the extended real-valued metric space $(\Lp[K], \mebr_{[\dK]})$ stays 
		related to the linear space $(\Lp, \norm{\cdot}_{\Lp})$ in terms of the topology and bornology induced by $\mebr$, 
		resp. those inherited by
		$\norm{\cdot}_{\Lp}$.
		\item
		The closedness of $K \subseteq \R^M$ is crucial in showing that $\W[K]$ is a sequentially closed subset of the linear space $\W$.
		This closedness property acts as a kind of replacement for the, a priori not available, notion of completeness with respect to the 
		``space'' $(\W[K], \mebr, \Reg)$.
	\end{itemize}
	For $l=0$, $k=N$ note in the latter item that
	equipping $\W[K]$ with $\mebr_{[\dKt]}$ and $\Reg_{[\dKo]}$ does not even lead to an (extended real-valued) metric space,
	in contrast to the classical case $(K,\dK) = (\R^M, \dRM)$.
\end{remark}

We will use the following assumption:

\begin{assumption} \label{as:Setting}  
  Let \autoref{ass:1} hold, $v^0 \in \Lp<2>$ and let $\W<1>$ and the associated topology be as defined in \autoref{eq:ChooseW}.
  
  In addition we assume: 
  \begin{itemize}
   \item $\op: \W<1> \to \Lp<2>$ is well--defined and sequentially continuous with respect to the specified topology on $\W<1>$ and 
   \item For every $t > 0$ and $\alpha > 0$ the level sets 
      \begin{equation}\label{itm: A} 
      \text{level}_t(\F<v^0><\alpha>[\dKt, \dKo]) \defeq \{ w \in \W<1> \ : \  \F<v^0><\alpha>[\dKt,\dKo] \leq t \}  
      \end{equation}
    are sequentially pre-compact subsets of $W(\Omega_1, \R^{M_1})$.
   \item There exists a $\bar{t} > 0$ such that $\text{level}_{\bar{t}}(\F<v^0><\alpha>[\dKt, \dKo])$ is nonempty.
   \item Only those $v \in \Lp<2>$ are considered which additionally fulfill $\mebr[v][v^0]_{[\dKt]} < \infty$. 
  \end{itemize}
\end{assumption}

\begin{remark}
  The third condition is sufficient to guarantee $\F<v^0><\alpha>[\dKt, \dKo]) \not \equiv \infty$. In contrast the condition $v^0 \in \Lp<2>$, cf. 
  \autoref{def:functional}, might not be sufficient if $d_2$ is not equivalent to $\dRMt|_{K_2 \times K_2}$.
\end{remark}

\begin{lemma} \label{thm:F_and_its_summands_are_seq_weakly_closed} 
   Let \autoref{as:Setting} hold. 
   Then the  mappings $\mebr_{[\dKt]}$, $\Reg_{[\dKo]}$ and $\F[\dKt, \dKo]$ have the following continuity properties:
   \begin{enumerate}
     \item \label{enu:continuity_of_mebr}
       The mapping $\mebr_{[\dKt]}: \Lp<2> \times \Lp<2> \rarr [0, \pinfty]$ is sequentially lower semi-continuous,
       i.e. whenever sequences $\seq{\phi}$, $\seq{\nu}$ in $\Lp<2>$ converge to $\phi_* \in \Lp<2>$ and 
       $\nu_*\in \Lp<2>$, respectively, we have $\mebr[\phi_*][\nu_*]_{[\dKt]} \leq \limi \limits \mebr[\phi_n][\nu_n]_{[\dKt]}$.
     \item \label{enu:seq_lscty_of_R} 
       The functional $\Reg_{[\dKo]}: \W<1> \rarr [0,\infty]$ is sequentially lower semi-continuous, i.e. whenever a 
       sequence $(w_n)_{n \in \N}$ in $\W<1>$ converges to some $w_* \in \W<1>$ we have
       \begin{equation*}
          \Reg_{[\dKo]}(w_*)  \leq \limi \Reg_{[\dKo]}(w_n).
       \end{equation*}
     \item \label{enu:seq_lscty_of_F} 
       The functional $\F[\dKt,\dKo]: W(\Omega_1, K_1) \rarr [0,\infty]$ is sequentially lower semi-continuous.
   \end{enumerate}
\end{lemma}

\begin{proof}
	\begin{enumerate}
	\item \label{enu:lsc_of_mebr}
		It is sufficient to show that for \emph{every} pair of sequences $\seq{\phi}$, $\seq{\nu}$ in $\Lp<2>$ which 
		converge to previously \emph{fixed} elements  $\phi_* \in \Lp<2>$ and $\nu_*\in \Lp<2>$, respectively, we can extract subsequences 
		$(\phi_{n_j})_{j \in \N}$ and $(\nu_{n_j})_{j \in \N}$, respectively, with 
		\begin{gather*}
		\mebr[\phi_*][\nu_*]_{[\dKt]} \leq \liminf_{j \rarr \infty} \mebr[\phi_{n_j}][\nu_{n_j}]_{[\dKt]}.
		\end{gather*}
		To this end let $(\phi_n)_{n \in \N},(\nu_n)_{n \in \N}$ be some sequences in $\Lp<2>$ with $\phi_n \rarr \phi_*$ and $\nu_n \rarr \nu_*$ in $\Lp<2>$. 
		\autoref{lem:Wsp_weakly_seq_closed_etc} ensures that there exist subsequences $(\phi_{n_j})_{j \in \N}, (\nu_{n_j})_{j \in \N}$ 
		converging to $\phi_*$ and $\nu_*$ pointwise almost everywhere, 
		which in turn implies $\big(\phi_{n_j}(\cdot), \nu_{n_j}(\cdot) \big) \to \big( \phi_*(\cdot), \nu_*(\cdot) \big)$ 
		pointwise almost everywhere. Therefrom, together with the continuity  of $\dKt: K_2 \times K_2 \to [0, \infty)$ 
		with respect to $\dRMt$, cf. \autoref{sec: Setting}, we obtain		
		by using the quadrangle inequality 
		that
		\begin{gather*}
		   | \dKt(\phi_{n_j}(x), \nu_{n_j}(x)) - \dKt(\phi_*(x), \nu_*(x)) |  \leq  \dKt(\phi_{n_j}(x), \phi_*(x))  +  \dKt(\nu_{n_j}(x), \nu_*(x)) \rarr 0,
	     \end{gather*}
		and hence		
		\begin{gather*}
		\dKt^{p_2}\big( \phi_{n_j}(x), \nu_{n_j}(x) \big) \to \dKt^{p_2} \big( \phi_*(x),\nu_*(x) \big) \text{ for almost every }
		x \in \Omega_2.
		\end{gather*}
		Applying Fatou's lemma we obtain
		\begin{gather*}
		\mebr[\phi_*][\nu_*]_{[\dKt]} = \int_{\Omega_2 }\limits \dKt^{p_2}( \phi_*(x),\nu_*(x) ) \dx
		\leq \liminf_{j \rarr \infty} \int_{\Omega_2}\limits \dKt^{p_2} ( \phi_{n_j}(x), \nu_{n_j}(x) ) \dx
		= \liminf_{j \rarr \infty} \mebr[\phi_{n_j}][\nu_{n_j}]_{[\dKt]}. 
		\end{gather*}
		\item 
		Let $(w_n)_{n \in \N}$ be a sequence in $\W<1>$ with $w_n \Wto w_*$ as $n \rarr \infty$.
		By \autoref{lem:Wsp_weakly_seq_closed_etc} there is a subsequence 
		$(w_{n_j})_{j \in \N}$ which converges to $w_*$ both in $\Lp<1>$ and pointwise almost everywhere.
		This further implies that
		\begin{gather*}
		\dKo^{p_1}\big( w_{n_j}(x), w_{n_j}(y) \big) \to \dKo^{p_1} \big( w_*(x),w_*(y) \big)
		\end{gather*}
		for almost every 
		\begin{equation}
		 \label{eq:A} 
		 (x,y) \in \Omega_1 \times \Omega_1 \supseteq \{(x,y) \in \Omega_1 \times \Omega_1 : x \neq y \} \eqdef A.
		\end{equation}
		Defining
		\begin{equation*}
		f_j(x,y) \defeq \left\{ \begin{array}{rcl}
		\frac{\dKo^{p_1}(w_{n_j}(x), w_{n_j}(y)) }{\normN[x-y]^{k+ps}} \rho^l(x-y) & \text{ for } & 
		(x,y) \in A,\\
		0 & \text{ for } & 
		(x,y) \in (\Omega_1 \times \Omega_1) \setminus A,\\
		\end{array} \right. \quad \text{ for all } j \in \N.
		\end{equation*}
		and 
		\begin{equation*}
		f_*(x,y) \defeq \left\{ \begin{array}{ccl}
		\frac{\dKo^{p_1}(w_*(x), w_*(y)) }{\normN[x-y]^{k+ps}}   \rho^l(x-y) & \text{ for } & 
		(x,y) \in A,\\
		0 & \text{ for } & 
		(x,y) \in (\Omega_1 \times \Omega_1) \setminus A\\
		\end{array} \right. 
		\end{equation*}
		we thus have $f_*(x,y) = \lim_{j \rarr \infty} f_j(x,y)$
		for almost every $(x,y) \in \Omega_1 \times \Omega_1$.
		Applying Fatou's lemma to the functions $f_j$ yields the assertion, due to the same reduction as in the proof of the first part.
		
		\item 
		It is sufficient to prove that the components $\mathcal{G}(\cdot) = \mebr[F(\cdot)][v]_{[\dKt]}$ and $\Reg = \Reg_{[\dKo]}$ of 
		$\F[\dKo,\dKt] = \mathcal{G} + \alpha \Reg$ are sequentially lower semi-continuous.
		To prove that $\mathcal{G}$ is sequentially lower semi-continuous in every $w_* \in W(\Omega_1, K_1)$
		let $(w_n)_{n \in \N}$ be a sequence in $W(\Omega_1, K_1)$ with $w_n \Wto w_*$ as $n \rarr \infty$.
		\autoref{as:Setting}, ensuring the sequential continuity of $\op: \W<1> \to \Lp<2>$, implies hence $\op[w_n] \rarr \op[w_*]$ in $\Lp<2>$ 
		as $n \rarr \infty$.
		By \autoref{enu:continuity_of_mebr} we thus obtain $\mathcal{G}(w_*) = \mebr[\op[w_*]][v] \leq \limi \mebr[\op[w_n]][v] = \limi \mathcal{G}(w_n)$. \\
		$\Reg$ is sequentially lower semi-continuous by \autoref{enu:seq_lscty_of_R}.
	\end{enumerate}
\end{proof}

\subsection{Existence of minimizers}

The proof of 
  the existence of a 
  minimizer of $\F[\dKt, \dKo]$ is along the lines of the proof 
in \cite{SchGraGroHalLen09}, taking into account \autoref{re:tricks}
We will need the following useful lemma, cf. \cite{SchGraGroHalLen09}, which links $\text{level}_t(\F<v^0><\alpha>)$ and 
$\text{level}_t(\F<v><\alpha>)$ for $\mebr[v][v^0] < \infty$.

\begin{lemma}\label{lem:Ineq_F_parameterchange_to_other_v} 
	It holds 
	\begin{align*}
		\F<v_\star><>[\dKt, \dKo](w)  & \leq   2^{p_2-1} \F<v_\diamond><>[\dKt, \dKo](w) + 2^{p_2 -1} \mebr[v_\diamond][v_\star]_{[\dKt]}^{p_2}
	\end{align*}
	for every 
	$w \in W(\Omega_1, K_1)$ and $v_\star, v_\diamond \in L^{p_2}(\Omega_2, K_2)$.
\end{lemma}

\begin{proof}
	Using the fact that for $p \geq 1$ we have that $|a+b|^p \leq 2^{p-1}(|a|^p + |b|^p), \ a,b \in \R \cup \{\infty\}$ and that $\mebr_{[\dKt]}$ fulfills the triangle inequality we obtain
	\begin{align*}
		\F<v_\star><>[\dKt, \dKo](w) 
		& = \mebr[\op[w]][v_\star]_{[\dKt]}^{p_2} + \alpha \Reg_{[\dKo]}(w) \\
		& \leq 2^{p_2-1} \big( \mebr[\op[w]][v_\diamond]_{[\dKt]}^{p_2} + \mebr[v_\diamond][v_\star]_{[\dKt]}^{p_2} \big) + \alpha \Reg_{[\dKo]}(w) \\
		& \leq 2^{p_2-1}\big( \F<v_\diamond><>[\dKt, \dKo](w) + \mebr[v_\diamond][v_\star]_{[\dKt]}^{p_2} \big).
	\end{align*}
\end{proof}

\begin{thm} \label{thm:F_dK_has_a_minimizer} 
Let \autoref{as:Setting} hold. Then the functional $\F<v><\alpha>[\dKt, \dKo]: W(\Omega_1, K_1) \rarr [0, \infty]$ 
attains a minimizer.
\end{thm}

\begin{proof}
  We prove the existence of a minimizer via the Direct Method.
  We shortly write $\F^v$ for $\F<v><\alpha>[\dKt, \dKo]$.
  Let $(w_n)_{n \in \N}$ be a sequence in $W(\Omega_1, K_1)$ with
  \begin{gather}\label{eq:w_n_is_minimizing_seq}
    \lim_{n \rarr \infty }\F^v(w_n) = \inf_{w \in W(\Omega_1, K_1)} \F^v(w).
  \end{gather}
  The latter infimum is not $\pinfty$, because $\F^v \equiv \pinfty$ would imply also $\F^{v^0} \equiv \pinfty$ due to 
  \autoref{lem:Ineq_F_parameterchange_to_other_v}, violating \autoref{as:Setting}.
  In particular there is some $c \in \R$ such that 
  $\F^v(w_n) \leq c$ for every $n \in \N$. 
  Applying \autoref{lem:Ineq_F_parameterchange_to_other_v} yields
  $\F^{v^0}(w_n) \leq 2^{p_2-1} \big( \F^{v}(w_n) + \mebr[v][v^0] \big)\leq 2^{p_2-1} \big( c + \mebr[v][v^0] \big) \eqdef\tilde{c} < \infty$ due to \autoref{as:Setting}.
  Since the level set $\text{level}_{\tilde{c}}(\F^{v^0})$ is sequentially pre-compact with respect to the topology given to $W(\Omega_1, \R^{M_1})$ we get the existence of a
  subsequence $(w_{n_k})_{k \in \N}$ which converges to some $w_* \in W(\Omega_1, \R^{M_1})$, 
  where actually $w_* \in W(\Omega_1, K_1)$ due to \autoref{lem:Wsp_weakly_seq_closed_etc}.
  Because $\F^v$ is sequentially lower semi-continuous, see \autoref{thm:F_and_its_summands_are_seq_weakly_closed}, we have 
  $\F^v(w_*) \leq \liminf_{k \rarr \infty} \F^v(w_{n_k})$. Combining this with 
  \autoref{eq:w_n_is_minimizing_seq} we obtain 
  \begin{gather*}
    \inf_{w \in W(\Omega_1, K_1)} \F^v(w) \leq \F^v(w_*) \leq \liminf_{k \rarr \infty} \F^v(w_{n_k}) 
    = \lim_{n\rarr \infty} \F^v(w_n) = \inf_{w \in W(\Omega_1, K_1)} \F^v(w).
  \end{gather*}
  In particular $\F^v(w_*) = \inf_{w \in W(\Omega_1, K_1)}\limits \F^v(w)$, meaning that $w_*$ is a minimizer of $\F^v$.
\end{proof}

In the following we investigate two examples, which are relevant for the numerical examples in \autoref{sec:Numerical_results}.
\begin{example}
\label{ex:coercive}
We consider that 
$W(\Omega_1,K_1) = W^{s, p_1}(\Omega_1, K_1)$  
with $p_1>1, \ 0 < s < 1$ and fix $k = N$. 

If the operator $\op$ is norm-coercive in the sense that the implication
    \begin{equation} \label{eq:coercive_F}
    \norm{w_n}_{L^{p_1}(\Omega_1, \R^{M_1})} \rarr \pinfty \Rightarrow \norm{\op[w_n]}_{L^{p_2}(\Omega_2, \R^{M_2})} \rarr \pinfty 
    \end{equation}
    holds true for every sequence $(w_n)_{n \in \N}$ in $W^{s,p_1}(\Omega_1, K_1)\subseteq W^{s,p_1}(\Omega_1, \R^{M_1})$, then the functional 
    \begin{equation*}
    \F[\dKt, \dKo] = \mebr[\op[w]][v]^{p_2}_{[\dKt]} + \alpha \Reg_{[\dKo]}(w): W^{s,p_1}(\Omega_1, K_1) \rarr [0, \infty]
    \end{equation*}
    is coercive. This can be seen as follows: 
    
    The inequality between $\dKo$ and $\dRMo|_{K_1 \times K_1}$ resp. $\dKt$ and $\dRMt|_{K_2 \times K_2}$, see \autoref{ass:1}, carries over to 
    $\F[\dKt,\dKo]$ and $\F[\dRMt|_{K_2 \times K_2},\dRMo|_{K_1 \times K_1}]$, i.e.
    \begin{gather*}
     \F[\dKt, \dKo] (w) \geq \F\big[\dRMt|_{K_2 \times K_2},\dRMo|_{K_1 \times K_1}\big](w) \text{ for all } w \in W^{s,p_1}(\Omega_1, K_1).
    \end{gather*} 
    Thus it is sufficient to show that $\F[\dRMt|_{K_2 \times K_2},\dRMo|_{K_1 \times K_1}]: W^{s,p_1}(\Omega_1, K_1) \rarr [0, \infty]$ is coercive:    
    To prove this we write shortly $\F$ instead of $\F[\dRMt|_{K_2 \times K_2},\dRMo|_{K_1 \times K_1}]$ and consider sequences
    $(w_n)_{n \in \N}$ in $W^{s,p_1}(\Omega_1, K_1)$ with $\norm{w_n}_{W^{s,p_1}(\Omega_1, \R^{M_1})} \rarr \pinfty$ as $n \rarr \infty$. We 
    show that $\F(w_n) \rarr +\infty$, as $n \rarr \infty$.
    Since  
    \begin{equation*} 
     \norm{w_n}_{W^{s,p_1}(\Omega_1, \R^{M_1})} =  \big( \norm{w_n}_{L^{p_1}(\Omega_1, \R^{M_1})}^{p_1} + \abs{w_n}_{W^{s,p_1}(\Omega_1, \R^{M_1})}^{p_1} \big)^{\frac{1}{p_1}}
    \end{equation*}
    the two main cases to be considered are $\norm{w_n}_{L^{p_1}(\Omega_1, \R^{M_1})} \rarr +\infty$ and 
    $\abs{w_n}_{W^{s,p_1}(\Omega_1, \R^{M_1})} \rarr +\infty$.
   
    \begin{enumerate}[label=\textbf{Case \arabic*}]
    \item \label{ex: coecivity_case1}
    $\norm{w_n}_{L^{p_1}(\Omega_1, \R^{M_1})} \rarr +\infty$. \\
    The inverse triangle inequality and the norm-coercivity of $\op$, \autoref{eq:coercive_F}, give
    $\norm{\op[w_n] - v}_{L^{p_2}(\Omega_2, \R^{M_2})} \geq \norm{\op[w_n]}_{L^{p_2}(\Omega_2, \R^{M_2})} - \norm{v}_{L^{p_2}(\Omega_2, \R^{M_2})} \rarr \pinfty$.
    Therefore also 
    \begin{equation*}
     \F(w_n) = \norm{\op[w_n] - v}^{p_2}_{L^{p_2}(\Omega_2, \R^{M_2})} 
      + \alpha \int\limits_{\Omega_1\times \Omega_1} \frac{\normMo[w_n(x) - w_n(y)]^{p_1}}{\normN[x-y]^{N+p_1 s}} \rho^l(x-y) \dxy
    \rarr +\infty.
    \end{equation*}

  \noindent
  \item \label{ex: coecivity_case2}
  $\abs{w_n}_{W^{s,p_1}(\Omega_1, \R^{M_1})} \rarr +\infty$. \\
  If $l=0$, then $\Reg_{[\dKo]}$ is exactly the $W^{s,p_1}(\Omega_1, \R^{M_1})$-semi-norm $|w|_{W^{s,p_1}(\Omega_1, \R^{M_1})}$ and we trivially 
  get the desired result. 
  
  Hence we assume from now on that $l = 1$. 
  The assumptions on $\rho$ ensure that there exists a  $\tau > 0$ and $\etau > 0$ such that
  \begin{align*}
    \stripet \defeq {}&  \{(x,y) \in \Omega_1 \times \Omega_1 : \rho(x-y) \geq \tau \}
  \\
    = {}& \{(x,y) \in \Omega_1 \times \Omega_1 : \normN[x-y] \leq \etau \},
  \end{align*}
  cf. \autoref{fig:stripe}.

  Splitting $\Omega_1 \times \Omega_1$ into $\stripet \eqdef \stripe$ and its complement 
  $(\Omega_1 \times \Omega_1) \setminus \stripet \eqdef \Sc$
  we accordingly split the integrals 
  $\abs{w_n}_{W^{s,p_1}(\Omega_1, \R^{M_1})} = \int\limits_{\Omega_1 \times \Omega_1} \frac{\normMo[w_n(x) - w_n(y)]^{p_1}}{\normN[x-y]^{N+p_1 s}} \dxy$ 
  and consider again two cases 
  $\int\limits_{\stripe} \frac{\normMo[w_n(x) - w_n(y)]^{p_1}}{\normN[x-y]^{N+p_1 s}} \dxy  \rarr +\infty$ and 
  $\int\limits_{\Sc} \frac{\normMo[w_n(x) - w_n(y)]^{p_1}}{\normN[x-y]^{N+p_1 s}} \dxy \rarr +\infty$, respectively.

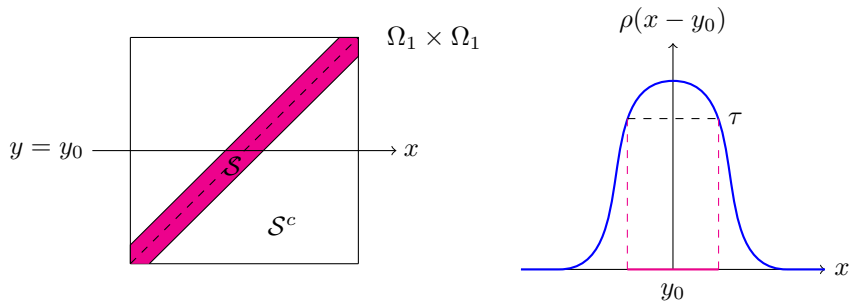
\begin{figure}[h!]
	\centering
	\begin{minipage}[t]{0.35\linewidth}
		\vspace*{-110pt}
		\begin{tikzpicture}
		\draw (0,0) -- (3,0);
		\draw (3,0) -- (3,3);
		\draw (3,3) -- (0,3);
		\draw (0,3) -- (0,0);
		\draw [ fill = magenta] (0.25,0) -- (3, 2.75) --(3,3)--(2.75, 3) -- (0,0.25) --(0,0);
		\draw [dashed] (0,0) -- (3,3);
		\node at (1.35,1.3) {$\stripe$};
		\node at (2.0,0.5) {$\Sc$};
		\node at (4,3) {$\Omega_1 \times \Omega_1$};
		\draw[->] (-0.5,1.5) -- (3.5,1.5);
		\node at (3.7,1.5){$x$};
		\node at (-1.1,1.5){$y=y_0$};
		\end{tikzpicture}
	\end{minipage}
	\hspace{1cm}
	\begin{minipage}[t]{0.35\linewidth}
		\begin{tikzpicture}
		\draw[->] (-2,0) -- (2,0);
		\node[right] at (2,0) {$x$};
		\draw[->](0,0) -- (0,3);
		\node[above] at (0,3) {$\rho(x-y_0)$};
		\draw[ thick, blue, out=0, in=180] (0,2.5) to (1.5,0);
		\draw[ thick, blue, out=180, in=0] (0,2.5) to (-1.5,0);
		\draw[thick, blue] (1.5,0) -- (2,0);
		\draw[thick, blue] (-1.5,0) -- (-2,0);	
		\draw[dashed] (-0.6,2) -- (0.6,2);		
		\node[right] at (0.6,2) {$\tau$};
		\draw[thick, magenta] (-0.6,0) -- (0.6,0);		
		\draw[dashed, magenta] (-0.6,0) to (-0.6,2);
		\draw[dashed, magenta] (0.6,0) to (0.6,2);
		\node at (0,-0.3){$y_0$};
		\end{tikzpicture}
	\end{minipage}
	\caption{The stripe $\stripe = \stripet$ if $\Omega_1$ is an open interval and its connection to the radial mollifier $\rho$ for fixed $y \in \Omega_1$.}
	\label{fig:stripe}
\end{figure}
  
  \begin{enumerate}[label=\textbf{\ref{ex: coecivity_case2}.\arabic*}]
  \item	
  $\int\limits_{\stripe} \frac{\normMo[w_n(x) - w_n(y)]^{p_1}}{\normN[x-y]^{N+p_1 s}} \dxy \rarr + \infty$. \\
  By definition of $\stripe$ we have $\rho(x-y) \geq \tau > 0$ for all $(x,y) \in \stripe$.
  Therefore
  \begin{gather*}
    \int\limits_{\stripe} \frac{\normMo[w_n(x) - w_n(y)]^{p_1}}{\normN[x-y]^{N+p_1 s}} \rho(x-y) \dxy 
    \geq \tau \int\limits_{\stripe} \frac{\normMo[w_n(x) - w_n(y)]^{p_1}}{\normN[x-y]^{N+p_1 s}} \dxy
    \rarr \pinfty.
  \end{gather*}
  Since $\alpha > 0$, it follows 
  \begin{align*}
    \F(w_n) &= \norm{\op(w_n) - v}^{p_2}_{L^{p_2}(\Omega_2, \R^{M_2})}
      + \underbrace{\alpha \int\limits_{\stripe} \frac{\normMo[w_n(x) - w_n(y)]^{p_1}}{\normN[x-y]^{N+p_1 s}} \rho(x-y) \dxy 
         }_{\rarr +\infty} \\
    & +  \underbrace{ \alpha \int\limits_{\Sc} \frac{\normMo[w_n(x) - w_n(y)]^{p_1}}{\normN[x-y]^{N+p_1 s}} \rho(x-y) \dxy 
         }_{\geq 0}
    \rarr +\infty.
  \end{align*}

  \noindent
  \item	\label{ex: coecivity_case22}
  $\int\limits_{\Sc} \frac{\normMo[w_n(x) - w_n(y)]^{p_1}}{\normN[x-y]^{N+p_1 s}} \dxy \rarr + \infty$.\\
  For $(x, y) \in \Sc$ it might happen that $\rho(x-y) = 0$, and thus instead of proving 
  $\F(w_n) \geq \int\limits_{\Sc} \frac{\normMo[w_n(x) - w_n(y)]^{p_1}}{\normN[x-y]^{N+p_1 s}} \rho(x-y) \dxy \rarr \pinfty$,
  as in Case 2.1, we rather show that $\F(w_n) \geq \norm{\op[w_n] - v}^{p_2}_{L^{p_2}(\Omega_2, \R^{M_2})} \rarr \pinfty$.
  For this it is sufficient to show that for every $c > 0$ there is some $C \in \R$ such that the implication
  \begin{gather*}
     \norm{\op[w]-v}^{p_2}_{L^{p_2}(\Omega_2, \R^{M_2})}\leq c
     \implies
      \int\limits_{\Sc} \frac{\normMo[w(x) - w(y)]^{p_1}}{\normN[x-y]^{N+p_1 s}}  \dxy \leq C,
  \end{gather*}
  holds true for all $w \in W^{s,p_1}(\Omega_1, K_1) \subseteq W^{s,p_1}(\Omega_1, \R^{M_1})$.
  To this end let $c > 0$ be given and consider an arbitrarily chosen $w \in W^{s,p_1}(\Omega_1, K_1)$ 
  fulfilling $\norm{\op[w] - v}^{p_2}_{L^{p_2}(\Omega_2, \R^{M_2})} \leq c$.
  
  Then $\norm{\op[w] - v}_{L^{p_2}(\Omega_2, \R^{M_2})} \leq \sqrt[p_2]{c}$. Using the triangle inequality and the monotonicity
  of the function $h: t \mapsto t^{p_2}$ on $[0, \pinfty)$ we get further
  \begin{align}\label{eq: Norm_estimate}
    \norm{\op[w]}^{p_2}_{L^{p_2}(\Omega_2, \R^{M_2})} 
    &= \norm{\op[w] - v + v}^{p_2}_{L^{p_2}(\Omega_2, \R^{M_2})} \nonumber\\
    &\leq \left( \norm{\op[w] - v}_{L^{p_2}(\Omega_2, \R^{M_2})} + \norm{v}_{L^{p_2}(\Omega_2, \R^{M_2})} \right)^{p_2} \nonumber\\
    & \leq \big( \sqrt[p_2]{c} + \norm{v}_{L^{p_2}(\Omega_2, \R^{M_2})} \big)^{p_2} \eqdef \tilde{c}. 
  \end{align}
  \noindent
  Due to the norm-coercivity, it thus follows 
  that $\norm{w}_{L^{p_1}(\Omega_1, \R^{M_1})} \leq \bar{c}$, $\bar{c}$ some constant. 
  Using \cite[Lemma 3.20]{SchGraGroHalLen09} it then follows that 
  \begin{gather}\label{eq: Convexity_Inequality}
    \normMo[w(x) - w(y)]^{p_1} \leq 2^{p_1-1} \normMo[w(x)]^{p_1} + 2^{p_1-1} \normMo[w(y)]^{p_1} 
  \end{gather}
  for all $(x,y) \in \Omega_1 \times \Omega_1$.
  Using \autoref{eq: Convexity_Inequality}, Fubini's Theorem and \autoref{eq: Norm_estimate} we obtain
  \begin{align*}
    \int\limits_{\Omega_1 \times \Omega_1} \normMo[w(x) - w(y)]^{p_1} \dxy
    & \leq \int\limits_{\Omega_1 \times \Omega_1}  2^{p_1-1} \normMo[w(x)]^{p_1} + 2^{p_1-1}\normMo[w(y)]^{p_1} \dxy \\
    &= \abs{\Omega_1} \int\limits_{\Omega_1}  2^{p_1-1} \normMo[w(x)]^{p_1} \dx + \abs{\Omega_1} \int\limits_{\Omega_1} 2^{p_1-1}\normMo[w(y)]^ {p_1} \dy \\
    &= 2\abs{\Omega_1} \int\limits_{\Omega_1}  2^{p_1-1} \normMo[w(x)]^{p_1}  \dx \\
    & = 2^{p_1} \abs{\Omega_1} \; \norm{w}^{p_1}_{L^{p_1}(\Omega_1, \R^{M_1})} \leq 2^{p_1} \abs{\Omega_1} \bar{c}^{p_1}.
  \end{align*}

  Combining $\normN[x-y] \geq \etau > 0$ for all $(x,y) \in \Sc$ with the previous inequality we obtain the needed estimate
  \begin{align*}
    \int\limits_{\Sc} \frac{\normMo[w(x) - w(y)]^{p_1}}{\normN[x-y]^{N+p_1 s}}  \dxy
    & \leq \frac{1}{\etau^{N+p_1 s}} \int\limits_{\Sc} \normMo[w(x) - w(y)]^{p_1}  \dxy 
  \\
    & \leq \frac{1}{\etau^{N+p_1 s}} \int\limits_{\Omega_1 \times \Omega_1} \normMo[w(x) - w(y)]^{p_1}  \dxy
  \\
    & \leq \frac{2^{p_1} \abs{\Omega_1} \bar{c}^{p_1}}{\etau^{N+p_1 s}} \eqdef C. 
  \end{align*}
  \end{enumerate}
  \end{enumerate}
  \end{example}
  
  The second example concerns the coercivity of $\F[\dKt,\dKo]$, defined in \autoref{eq:functional}, when 
  $\op$ denotes the \emph{masking operator} occurring in image inpainting. To prove this result we require the following auxiliary lemma:
  \begin{lemma}\label{lem. auxLemma}
     There exists a constant $C \in \R $ such that for all $w \in W^{s,p_1}(\Omega_1, \R^{M_1}), \ 0<s < 1, \ l \in \{0,1\}, \ 1 < p_1 < \infty$ and
     $D \subsetneq \Omega_1$ nonempty such that
  \begin{equation} \label{eq: coercivity_inpainting}
  \norm{w}_{L^{p_1}(D, \R^{M_1})}^{p_1} \leq C \left( \norm{w}_{L^{p_1}(\Omega_1 \setminus D, \R^{M_1})}^{p_1} + 
  \int\limits_{\Omega_1 \times \Omega_1} \frac{\normMo[w(x) - w(y)]^{p_1}}{\normN[x-y]^{N+p_1 s}} \rho^l(x-y) \dxy  \right).
  \end{equation}
  \end{lemma}
  \begin{proof}
   The proof is inspired by the proof of Poincaré's inequality in \cite{Eva10}. It is included here for the sake of completeness.

   Assume first that $l=1$. Let $\stripe$ be as above, 
   \begin{align*}
   \stripe \defeq {}&  \{(x,y) \in \Omega_1 \times \Omega_1 : \rho(x-y) \geq \tau \}
   \\
   = {}& \{(x,y) \in \Omega_1 \times \Omega_1 : \normN[x-y] \leq \eta \}.
   \end{align*} 
   If the stated inequality \autoref{eq: coercivity_inpainting} would be false, then for every $n \in \N$ there would exists a 
   function $w_n \in W^{s,p_1}(\Omega_1, \R^{M_1})$ satisfying
   \begin{equation}\label{eq: inpaintingContraEquation}
   \norm{w_n}_{L^{p_1}(D, \R^{M_1})}^{p_1} \geq n \big( \norm{w_n}_{L^{p_1}(\Omega_1 \setminus D, \R^{M_1})}^{p_1} + 
   \int\limits_{\Omega_1 \times \Omega_1} \frac{\normMo[w_n(x) - w_n(y)]^{p_1}}{\normN[x-y]^{N+p_1 s}} \rho(x-y) \dxy  \big).
   \end{equation}
   By normalizing we can assume without loss of generality 
   \begin{enumerate}
   	\item $\norm{w_n}_{L^{p_1}(D, \R^{M_1})}^{p_1} = 1$. \label{itm: inpainting1}
   \end{enumerate}
   Moreover, by \autoref{eq: inpaintingContraEquation}
   \begin{enumerate}
   	\setcounter{enumi}{1}
   	\item  $\norm{w_n}_{L^{p_1}(\Omega_1 \setminus D, \R^{M_1})}^{p_1} < \frac{1}{n}$, \label{itm: inpainting2}
   	\item $\int\limits_{\Omega_1 \times \Omega_1} \frac{\normMo[w_n(x) - w_n(y)]^{p_1}}{\normN[x-y]^{N+p_1 s}} \rho(x-y) \dxy < \frac{1}{n}$. \label{itm: inpainting3}
   \end{enumerate}
   
   By \autoref{itm: inpainting1} and \autoref{itm: inpainting2} we get that 
   $ \norm{w_n}_{L^{p_1}(\Omega_1, \R^{M_1})}^{p_1} = \norm{w_n}_{L^{p_1}(D, \R^{M_1})}^{p_1} + \norm{w_n}_{L^{p_1}(\Omega_1 \setminus D, \R^{M_1})}^{p_1} < 1 + \frac{1}{n} < 2 $
   is bounded. Moreover
   \begin{align*}
   \abs{w_n}^{p_1}_{W^{s,p_1}(\Omega_1, \R^{M_1})} 
   & = \int\limits_{\stripe} \frac{\normMo[w_n(x) - w_n(y)]^{p_1}}{\normN[x-y]^{N+p_1 s}} \dxy
   + \int\limits_{\stripe^c} \frac{\normMo[w_n(x) - w_n(y)]^{p_1}}{\normN[x-y]^{N+p_1 s}} \dxy \\
   & \leq \frac{1}{\tau} \int\limits_{\stripe} \frac{\normMo[w_n(x) - w_n(y)]^{p_1}}{\normN[x-y]^{N+p_1 s}} \rho(x-y) \dxy
   + \frac{2^p_1 \abs{\Omega_1}}{\eta^{N+p_1 s}}\norm{w_n}^{p_1}_{L^{p_1}(\Omega_1, \R^{M_1})} \\
   & < \frac{1}{\tau n} + \frac{2^{p_1+1} \abs{\Omega_1}}{\eta^{N+p_1 s}}
   \leq \frac{1}{\tau} + \frac{2^{p_1+1} \abs{\Omega_1}}{\eta^{N+p_1 s}} \eqdef c < \infty,
   \end{align*}
   where $c$ is independent of $n$.
   This yields that the sequence  
   $(w_n)_{n \in \N}$ is bounded in $W^{s,p_1}(\Omega_1, \R^{M_1})$ by $(2 + c)^{\frac{1}{p_1}}$. 
   By the reflexivity of $\Wsp<1>[][\R^{M_1}]$ for $p_1 \in (1, \infty)$ and \autoref{lem:Wsp_weakly_seq_closed_etc} 
   there exists a subsequence $(w_{n_k})_{k \in \N}$ of $(w_n)_{n \in \N}$  
   and $w_* \in W^{s,p_1}(\Omega_1, \R^{M_1})$
   such that $w_{n_k} \rarr w^*$ strongly in $L^{p_1}(\Omega_1, \R^{M_1})$ and 
   pointwise almost everywhere.

   Using the continuity of the norm and dominated convergence we obtain
   \begin{enumerate}
   	\item $\norm{w^*}_{L^{p_1}(D, \R^{M_1})}^{p_1} = 1$, in particular $w^*$ is not the null-function on D,
   	\item $\norm{w^*}_{L^{p_1}(\Omega_1 \setminus D, \R^{M_1})}^{p_1} = 0$ since $n \in \N$ is arbitrary and hence $w^* \equiv 0$ on $\Omega_1 \setminus D$.
   	\item \begin{equation*}
   	\liminf_{n \rarr \infty} \frac{1}{n} 
   	> \liminf_{n \rarr \infty} \int\limits_{\stripe} \frac{\normMo[w_n(x) - w_n(y)]^{p_1}}{\normN[x-y]^{N+p_1 s}} \rho(x-y) \dxy 
   	\geq \frac{\tau}{\eta^{N+p_1 s}} \int\limits_{\stripe} \normMo[w^*(x) - w^*(y)]^{p_1},
   	\end{equation*}
   	i.e. $w^*(x) = w^*(y) $ for $(x,y) \in \stripe$ yielding that $w^*$ locally constant and hence even constant since $\Omega_1$ is connected,
   \end{enumerate}
   which gives the contradiction. 
   
   In the case $l=0$ we use similar arguments, where the distance $\normN[x-y]$ in the last inequality can be estimated by $\mathrm{diam}|\Omega_1|$ (instead of $\eta$) since $\Omega_1$ is bounded.
  \end{proof}
  
  \begin{remark}
  	In case $l=1$ it follows that the sharper inequality holds true: 
  	There exists a constant $C \in \R $ such that for all $w \in W^{s,p_1}(\Omega_1, \R^{M_1}), \ 0<s < 1, \ 1 < p_1 < \infty$ and $D \subsetneq \Omega_1$ nonempty such that
  	\begin{equation}\label{eq: inpaintingIneq2}
  	\norm{w}_{L^{p_1}(D, \R^{M_1})}^{p_1} \leq C \left( \norm{w}_{L^{p_1}(\Omega_1 \setminus D, \R^{M_1})}^{p_1} + 
  	\int\limits_{\stripe} \frac{\normMo[w(x) - w(y)]^{p_1}}{\normN[x-y]^{N+p_1 s}} \rho^l(x-y) \dxy  \right).
  	\end{equation} 
  \end{remark}
  
  \begin{example} 
  \label{ex:in}
  As in \autoref{ex:coercive} we consider that 
  $W(\Omega_1,K_1) = W^{s, p_1}(\Omega_1, K_1)$  
  with $p_1>1, \ 0 < s < 1$ and fix $k = N$. 
  
  Assume that $\op$ is the inpainting operator, i.e. 
  \begin{equation*}
    \op(w) = \chi_{\Omega_1 \backslash D} (w),
  \end{equation*} 
  where $D \subseteq \Omega_1, \ w \in W^{s,p_1}(\Omega_1, K_1)$. Since the dimension of the data $w$ and the image data $\op(w)$ have the same dimension at every point $x \in \Omega_1$, we write $M \defeq M_1 = M_2$. \\
  Then the functional 
  \begin{equation*}
  \F[\dKt, \dKo] = \mebr[\op[w]][v]^{p_2}_{[\dKt]} + \alpha \Reg_{[\dKo]}(w): W^{s,p_1}(\Omega_1, K_1) \rarr [0, \infty]
  \end{equation*}
  is coercive for $p_2 \geq p_1$: \\
  The fact that $p_2 \geq p_1$ and that $\Omega_1$ is bounded ensures that
  \begin{equation}\label{eq: LpEmbedding}
  L^{p_2}(\Omega_1 \backslash D, \R^M) \subseteq L^{p_1}(\Omega_1 \backslash D, \R^M).
  \end{equation} 
  The proof is done
  using the same arguments as in the proof of \autoref{ex:coercive}, where we additionally split \ref{ex: coecivity_case1} into the two sub-cases 
  \begin{enumerate}[label=\textbf{\ref{ex: coecivity_case1}.\arabic*}]
  	\item \label{ex: coecivity_case11}
  	$\norm{w_n}_{L^{p_1}(D, \R^M)} \rarr \pinfty$
  	\item \label{ex: coecivity_case12} 
  	$\norm{w_n}_{L^{p_1}(\Omega_1 \setminus D, \R^M)} \rarr \pinfty$
  \end{enumerate}
   and using additionally \autoref{lem. auxLemma}, \autoref{eq: inpaintingIneq2} and \autoref{eq: LpEmbedding}.    
\end{example}

\section{Stability and Convergence} \label{sec: Stability_and_Convergence}
In this section we will first show a stability and afterwards a convergence result. We use the notation introduced in 
\autoref{sec:  Setting}. In particular $W(\Omega_1, K_1)$ is as defined in \autoref{eq:ChooseW}. 
We also stress that we use 
notationally simplified versions $\F<v>$ of $\F<v><\alpha>[\dKt, \dKo]$ and $\Reg$ of $\Reg_{[\dKo]}$ whenever possible.
See \autoref{eq: functional_with_some_metric}, \autoref{eq:d2} and \autoref{eq:d3}.

\begin{thm} \label{thm:Stability}
Let \autoref{as:Setting} be satisfied.
Let $v^\delta \in L^{p_2}(\Omega_2, K_2)$ and let $(v_n)_{n \in \N}$ be a sequence in $L^{p_2}(\Omega_2, K_2)$ such that
$\mebr[v_n][v^\delta]_{[\dKt]} \rarr 0$.
Then every sequence $\seq{w}$ with
  \begin{equation*}
    w_n \in \arg \min \{ \F<v_n><\alpha>[\dKt, \dKo](w) \ : \ w \in W(\Omega_1, K_1) \}
  \end{equation*}	
  has a converging subsequence w.r.t. the topology of $W(\Omega_1, K_1)$.
  The limit $\tilde{w}$ of any such converging subsequence $(w_{n_k})_{k \in \N}$ is a minimizer of 
  $\F^{v^\delta}[\dKt, \dKo]$.
  Moreover, $(\Reg(w_{n_k}))_{k \in \N}$ converges to $\Reg(\tilde{w})$.
\end{thm}

The subsequent proof of \autoref{thm:Stability} is similar to the proof of \cite[Theorem 3.23]{SchGraGroHalLen09}.

\begin{proof}
For the ease of notation we simply write $\F<v^\delta>$ instead of $\F<v^\delta><\alpha>[\dKt, \dKo]$ and 
$\mebr[v][\tilde{v}] = \mebr[v][\tilde{v}]_{[\dKt]}$ 

By assumption the sequence $(\mebr[v_n][v^\delta])_{n\in \N}$ converges to $0$ and thus is bounded, i.e., 
there exists $B \in (0, \pinfty)$ such that 
\begin{gather} \label{eq:seq_bounded_IN_stability_proof}
  \mebr[v_n][v^\delta] \leq B \text{ for all } n \in \N.
\end{gather}
Because $w_n \in \arg \min \{ \F<v_n>(w)  :  w \in W(\Omega_1, K_1) \}$ it follows that 
\begin{equation}\label{eq: w_n Minimizer}
  \F<v_n>(w_n) \leq \F<v_n>(w) \text{ for all } w \in W(\Omega_1, K_1).
\end{equation}
By \autoref{as:Setting} there is a $\overline{w} \in W(\Omega_1, K_1)$ such that $\F^{v^0}(\overline{w}) <\infty$. Set $c \defeq 2^{p_2-1}$.  
Using \autoref{as:Setting} and applying \autoref{lem:Ineq_F_parameterchange_to_other_v}, 
\autoref{eq: w_n Minimizer} and \autoref{eq:seq_bounded_IN_stability_proof} implies that for all $n \in \N$
\begin{align*}
  \F<v^\delta> (w_n)  
  & \leq c \F<v_n>(w_n) + c \mebr[v_n][v^\delta]^{p_2}   
\\
  & \leq c \F<v_n>(\overline{w}) + c B^{p_2}
\\
  & \leq c \big[c \F<v^\delta>(\overline{w}) + c \mebr[v^\delta][v_n]^{p_2}  \big] + cB^{p_2}
\\
  & \leq c^2 \F<v^\delta>(\overline{w}) + (c^2 + c)B^{p_2}
\\
  & \leq c^3 \big( \F<v^0>(\overline{w}) + \mebr[v^0][v^\delta] \big) + (c^2 + c)B^{p_2} \eqdef m < \infty.
\end{align*}
Applying again \autoref{lem:Ineq_F_parameterchange_to_other_v} we obtain
$\F<v^0> (w_n) \leq c \F<v^\delta> (w_n) + c \mebr[v^\delta][v^0]^{p_2} \leq m + c \mebr[v^\delta][v^0]^{p_2} \eqdef \widetilde m < \infty$.
Hence, from item \eqref{itm: A} it follows that the sequence $\seq{w}$
contains a converging subsequence.

Let now $(w_{n_k})_{k \in \N}$ be an arbitrary subsequence of $\seq{w}$ which converges in $W(\Omega_1, K_1)$ to some 
$\tilde w \in W(\Omega_1, \R^{M_1})$. Then, from \autoref{lem:Wsp_weakly_seq_closed_etc} and the continuity properties of $\op$ 
it follows that $\tilde w \in W(\Omega_1, K_1)$ and $(\op[w_{n_k}], v_{n_k}) \rarr (\op[\tilde w], v^\delta)$ in 
$L^{p_2}(\Omega_2, K_2) \times L^{p_2}(\Omega_2, K_2)$. Moreover, using \autoref{thm:F_and_its_summands_are_seq_weakly_closed},
\autoref{eq: w_n Minimizer} 
and the triangle inequality 
it follows that for every $w \in W(\Omega_1, K_1)$ the following estimate holds true
\begin{align*}
  \F<v^\delta>(\tilde w) 
  & = \mebr[\op(\tilde w)][v^\delta]^{p_2} + \alpha \Reg(\tilde w)
    \leq \mebr[\op(\tilde w)][v^\delta]^{p_2} + \alpha \liminf_{k \rarr \infty} \Reg(w_{n_k})
    \leq \mebr[\op(\tilde w)][v^\delta]^{p_2} + \alpha \limsup_{k \rarr \infty} \Reg(w_{n_k})
\\& 
    \leq \liminf_{k \rarr \infty}  \mebr[\op(w_{n_k})][v_{n_k}]^{p_2} + \alpha \limsup_{k \rarr \infty} \Reg(w_{n_k})    
    \leq 
    \limsup_{k \rarr \infty} \F<v_{n_k}>(w_{n_k})
    \leq \limsup_{k \rarr \infty} \F<v_{n_k}>(w)
\\& 
= \left( \limsup_{k \to \infty} \mebr[F(w)][v_{n_k}] \right)^{p_2} + \alpha \Reg(w)
    \leq \left( \limsup_{k \to \infty} \big(\mebr[F(w)][v^\delta] + \mebr[v^\delta][v_{n_k}] \big) \right)^{p_2} + \alpha \Reg(w)
\\& = \F<v^\delta>(w).
\end{align*}

This shows that $\tilde w$ is a minimizer of $\F<v^\delta>$.
Choosing $w = \tilde w$ in the previous estimate we obtain the equality
\begin{gather*}
  \mebr[\op(\tilde w)][v^\delta]^{p_2}  + \alpha \Reg(\tilde w) 
  = \mebr[\op(\tilde w)][v^\delta]^{p_2} + \alpha \liminf_{k \rarr \infty} \Reg(w_{n_k}) 
  = \mebr[\op(\tilde w)][v^\delta]^{p_2} + \alpha \limsup_{k \rarr \infty} \Reg(w_{n_k}) \,.
\end{gather*}
Due to $\mebr[\op(\tilde w)][v^\delta]^{p_2} \leq \F<v^\delta>(\tilde w) \leq m < \infty$ this gives
\begin{gather*}
  \Reg(\tilde w) = \lim_{k \rarr \infty} \Reg(w_{n_k}).
\end{gather*}
\end{proof}

Before proving the next theorem we need the following definition, cf. \cite{SchGraGroHalLen09}.
\begin{definition}
  Let $v^0 \in \Lp<2>$.
  Every element $w^* \in W(\Omega_1, K_1)$ fulfilling
  \begin{align} \label{eq: R_minimizing_solution}
  \begin{split}	
  &\op[w^*] = v^0 \\ 
  &\Reg(w^*) = \min \{ \Reg(w) \ : \ w \in W(\Omega_1, K_1), \ \op[w] = v^0 \}.
  \end{split}
  \end{align}
  is called an 
  \emph{$\Reg$-minimizing solution} of the equation $\op[w] = v^0$ or shorter just
  \emph{$\Reg$-minimizing solution}.
\end{definition}

The following theorem and its proof are inspired by \cite[Theorem 3.26]{SchGraGroHalLen09}. 

\begin{thm} \label{thm: convergence}
  Let \autoref{as:Setting} be satisfied.	  
  Let there exist an $\Reg$-minimizing solution $w^\dagger \in W(\Omega_1, K_1)$ and 
  let $\alpha: (0, \infty) \rarr (0,\infty)$ be a function satisfying
  \begin{equation}\label{eq: assumptions_on_alpha}
    \alpha(\delta) \rarr 0 \text{ and } \frac{\delta^{p_2}}{\alpha(\delta)} \rarr 0
	\text{ for } \delta \to 0.
  \end{equation} 
  Let $(\delta_n)_{n \in \N}$ be a sequence of positive real numbers converging to $0$. Moreover, let 
  $(v_n)_{n \in \N}$ be a sequence in $L^{p_2}(\Omega_2, K_2)$ with $\mebr[v^0][v_n]_{[\dKt]} \leq \delta_n$ and 
  set $\alpha_n \defeq \alpha(\delta_n)$.
  
  Then every sequence $\seq{w}$ of minimizers 
  \begin{equation*}
   w_n \in \arg \min \{ \F^{v_n}_{\alpha_n}[\dKt, \dKo](w) \ : \ w \in W(\Omega_1, K_1) \}
  \end{equation*}
  has a converging subsequence $w_{n_k} \Wto \tilde{w}$ as $k \to \infty$, and the limit $\tilde{w}$ is always an $\Reg$-minimizing solution. 
  In addition, $\Reg(w_{n_k}) \rarr \Reg(\tilde{w})$.
	
  Moreover, if $w^\dagger$ is unique it follows that $w_n \Wto w^\dagger$ and $\Reg(w_{n}) \rarr \Reg(w^\dagger)$.
\end{thm}

\begin{proof}
  We write shortly $\mebr$ for $\mebr_{[\dKt]}$.
  Taking into account that $w_n \in \argmin \{ \F^{v_n}_{\alpha_n}[\dKt, \dKo](w) \ : \ w \in W(\Omega_1, K_1) \}$ it follows that
  \begin{gather*}
    \mebr[\op[w_n]][v_n]^{p_2} \leq \F<v_n><\alpha_n>(w_n) \leq \F<v_n><\alpha_n>(w^\dagger) = 
    \mebr[v^0][v_n]^{p_2} + \alpha_n \Reg(w^\dagger) \leq \delta_n^{p_2} + \alpha_n \Reg(w^\dagger) \rarr 0,
  \end{gather*}
  yielding $\mebr[\op[w_n]][v_n] \rarr 0$ as $n \rarr \infty$.
  The triangle inequality gives $\mebr[\op[w_n]][v^0] \leq \mebr[\op[w_n]][v_n] + \mebr[v_n][v^0] \rarr 0$ as 
  $n \rarr \infty$ and
  \autoref{re:tricks} ensures 
  $\norm{F(w_n) - v^0}_{\Lp<2>[][\R^{M_2}]} \leq \mebr[\op[w_n]][v^0] \rarr 0$ as 
  $n \rarr \infty$, so that  
  \begin{gather}\label{eq:Convergence_of_operator}
    \op[w_n] \rarr v^0 \text{ in } L^{p_2}(\Omega_2, \R^{M_2}).
  \end{gather} 
  Since
  \begin{gather*}
    \Reg(w_n)  \leq \frac{1}{\alpha_n} \F<v_n><\alpha_n>(w_n) \leq \frac{1}{\alpha_n} \F<v_n><\alpha_n>(w^\dagger) 
    = \frac{1}{\alpha_n}\big( \mebr[v^0][v_n]^{p_2} + \alpha_n \Reg(w^\dagger)  \big) \leq \frac{\delta_n^{p_2}}{\alpha_n} + \Reg(w^\dagger),
  \end{gather*}
  we also get 
  \begin{gather}\label{eq:Regularizer_values_bounded}
    \limsup_{n \rarr \infty} \Reg(w_n) \leq \Reg(w^\dagger).
  \end{gather}
  Set $\alpha_{\mathrm{max}} \defeq \max\{\alpha_n : n \in \N\}$. 
  Since 
  \begin{gather*}
    \limsup_{ n \rarr \infty } \F<v^0><\alpha_n>(w_{n}) \leq 
    \limsup_{n \rarr \infty } \big( \mebr[\op[w_{n}]][v^0]^{p_2} + \alpha_{\mathrm{max}} \Reg(w_{n})  \big) \leq \alpha_{\mathrm{max}} \Reg(w^\dagger)
  \end{gather*}
 the sequence $\F<v^0><\alpha_{\mathrm{max}}>(w_{n})$ is bounded. From \autoref{as:Setting}, item \eqref{itm: A} it follows that there exists 
 a converging subsequence $(w_{n_k})_{k \in \N}$ of $\seq{w}$. The limit of $(w_{n_k})_{k \in \N}$ is denoted by 
 $\tilde{w}$. Then, from \autoref{lem:Wsp_weakly_seq_closed_etc} it follows that $\tilde{w} \in W(\Omega_1, K_1)$.
 Since the operator $\op$ is sequentially continuous it follows that $\op[w_{n_k}] \rarr \op[\tilde{w}]$ in $L^{p_2}(\Omega_2, K_2)$.	
 This shows that actually $\op[\tilde{w}] = v^0$ since \autoref{eq:Convergence_of_operator} is valid.
 Then, from \autoref{thm:F_and_its_summands_are_seq_weakly_closed} it follows that the functional 
 $\Reg: W(\Omega_1, K_1) \rarr [0, \pinfty]$ is sequentially lower semi-continuous,
	so that $\Reg(\tilde{w}) \leq \liminf_{k \rarr \infty} \Reg(w_{n_k})$.
	Combining this with \autoref{eq:Regularizer_values_bounded} we also obtain 
	$$ \Reg(\tilde{w}) \leq \liminf_{k \rarr \infty} \Reg(w_{n_k}) \leq \limsup_{k \rarr \infty} \Reg(w_{n_k}) \leq \Reg(w^\dagger) \leq \Reg(\tilde{w}),$$
	using the definition of $w^\dagger$. 
	This, together with the fact that $\op[\tilde{w}] = v^0$ we see that $\tilde{w}$ is an $\Reg$-minimizing solution and that 
	$\lim_{k \rarr \infty} \Reg(w_{n_k})= \Reg(\tilde{w})$. 
	
	Now assume that the solution fulfilling \autoref{eq: R_minimizing_solution} is unique; we call it $w^\dagger$.  
	In order to prove that $w_n \Wto w^\dagger$ it is sufficient to show that any subsequence
	has a further subsequence converging to $w^\dagger$, cf.
	\cite[Lemma 8.2]{SchGraGroHalLen09}.
	Hence, denote by $(w_{n_k})_{k \in \N}$ an arbitrary subsequence of $(w_n)$, the sequence of minimizers.
	Like before we can show that $\F<v^0><\alpha>(w_{n_k})$ is bounded and we can extract a converging subsequence 
	$(w_{n_{k_l}})_{l \in \N}$. The limit of this subsequence is $w^\dagger$ since it is the unique solution fulfilling 
	\autoref{eq: R_minimizing_solution}, showing that $w_n \Wto w^\dagger$. Moreover, $w^\dagger \in W(\Omega_1, K_1)$. 
	Following the arguments above we obtain as well $\lim_{n \rarr \infty} \Reg(w_{n})= \Reg(w^\dagger).$ 
\end{proof}
\begin{remark}
\autoref{thm:Stability} guarantees that the minimizers of $\F<v_n><\alpha>[\dKt, \dKo]$ depend continuously on $v^\delta$ while 
\autoref{thm: convergence} ensures that they converge to a solution of $\op(w) = v^0$, $v^0$ the exact data, while $\alpha$ tends to zero. 
\end{remark}

\section{Discussion of the Results and Conjectures}
In this section we summarize some open problems related to double integral 
expressions
of functions with values on manifolds.

\subsection{Relation to single integral representations} 
In the following we show for one particular case of functions that have values in a manifold, that the double 
integral formulation $\Reg_{[\dKo]}$, defined in \autoref{eq:d3}, approximates a single energy integral. The basic 
ingredient for this derivation is the exponential map related to the metric $d_1$ on the manifold.
In the following we investigate manifold--valued functions $w \in W^{1,2}(\Omega, \mathcal{M})$, where we 
consider $\mathcal{M} \subseteq \R^{M \times 1}$ to be a connected, complete Riemannian manifold. 
In this case some of the regularization functionals $\Reg_{[\dKo]}$, defined in \autoref{eq:d3}, can be 
considered as approximations of \emph{single} integrals. In particular we aim to generalize \autoref{eq:double_integral} 
in the case $p=2$. 

We have that 
\begin{equation*}
 \nabla w = \begin{bmatrix} 
              \frac{\partial w_1}{\partial x_1} & \cdots & \frac{\partial w_1}{\partial x_N} \\
              \vdots & \ddots & \vdots\\
              \frac{\partial w_M}{\partial x_1} & \cdots & \frac{\partial w_M}{\partial x_N}
\end{bmatrix} \in \R^{M \times N}.
\end{equation*}
In the following we will write $\Reg_{[\dKo],\ve}$ instead of $\tfrac12 \Reg_{\dKo}$ to stress the dependence on $\ve$ in contrast to above;
the factor $\frac{1}{2}$ was added due to reasons of calculation.
Moreover, let 
$\hat{\rho} : \R_+ \to \R_+$ be in $C_c^\infty(\R_+, \R_+)$ and satisfy
\begin{equation*}
\abs{\mathbb{S}^{N-1}}\int_0^\infty \hat{t}^{N-1} \hat{\rho}\left(\hat{t}\right)d \hat{t} = 1\;.
\end{equation*}
Then for every $\ve > 0$
\begin{equation*}
x \in \R^n \mapsto \rho_\ve(x)\defeq \frac{1}{\ve^N} \hat{\rho}\left(\frac{\normN[x]}{\ve}\right)
\end{equation*}
is a mollifier, cf. \autoref{ex:mol}. \\
$\Reg_{[\dKo],\ve}$ 
(with $p_1=2$) then reads as follows: 
\begin{equation}
\label{eq:di_II}
 \Reg_{[\dKo],\ve}(w) 
 \defeq 
 \frac{1}{2}\int\limits_{\Omega\times \Omega} \frac{d_1^2(w(x),w(y))}{\normN[x-y]^2} \rho_\ve(x-y) \dxy\,.
\end{equation}
Substitution with spherical coordinates $y = x - t \theta \in \R^{N \times 1}$ with  
$\theta \in \mathbb{S}^{N-1} \subseteq \R^{N \times 1}$, $t \geq 0$ gives
\begin{equation}
\label{eq:reg2}
\begin{aligned}
\lim_{\ve \searrow 0}
\Reg_{[\dKo],\ve}(w) &= 
\lim_{\ve \searrow 0}
\frac{1}{\ve^N}
 \int\limits_{\Omega} \int\limits_{\mathbb{S}^{N-1}} 
 \int\limits_0^\infty \frac{1}{2} d_1^2(w(x),w(x-t \theta)) t^{N-3} \hat{\rho}\left(\frac{t}{\ve}\right) \mathrm{d}t \,\mathrm{d}\theta \dx\;.
\end{aligned}
\end{equation}
Now, using that for $m_1 \in \mathcal{M}$ fixed and $m_2 \in \mathcal{M}$ such that $m_1$ and $m_2$ are joined by a unique minimizing geodesic (see for instance \cite{FigVil11} where the concept of exponential mappings is explained) 
\begin{equation}\label{eq:partial_II}
 \frac{1}{2} \partial_2 d_1^2(m_1,m_2) = - (\exp_{m_2})^{-1}(m_1) \in \R^{M \times 1},
\end{equation}
where $\partial_2$ denotes the derivative of $d_1^2$ with respect to the second component. 
By application of the chain rule we get
\begin{equation*}
 \begin{aligned}
  - \frac{1}{2} \nabla_y d_1^2(w(x),w(y)) &= 
    \underbrace{(\nabla w(y))^T}_{\in \R^{N \times M}} \underbrace{(\exp_{w(y)})^{-1}(w(x))}_{\in \R^{M \times 1}}\in \R^{N \times 1}\;,
 \end{aligned}
\end{equation*}
where $w(x)$ and $w(y)$ are joined by a unique minimizing geodesic. This assumption seems reasonable due to the fact 
that we consider the case $\ve \searrow 0$. 
Let $\cdot$ denote the scalar multiplication of two vectors in $\R^{N \times 1}$, then the last equality shows that
\begin{equation*}
 \begin{aligned}
\frac{1}{2} d_1^2(w(x),w(x-t \theta)) 
&= - \frac{1}{2}  \left[  d_1^2\big(w(x),w( (x-t\theta)  + t \theta )\big) - d_1^2\big(w(x),w(x-t \theta)\big)  \right]  \\ 
&\approx \left( \left(\nabla w(x-t \theta)\right)^T (\exp_{w(x-t \theta)})^{-1}(w(x)) \right)  \cdot t\theta\;.   
\end{aligned}
\end{equation*}
Thus from \autoref{eq:reg2} it follows that 
\begin{equation}
\label{eq:reg3}
\begin{aligned}
~ & \lim_{\ve \searrow 0} 
\Reg_{[\dKo],\ve}(w) \\
\approx & 
\lim_{\ve \searrow 0}
 \frac{1}{\ve^N}
 \int\limits_{\Omega} \int\limits_{\mathbb{S}^{N-1}} 
 \int\limits_0^\infty \left( \left( \nabla w(x-t \theta)\right)^T (\exp_{w(x-t \theta)})^{-1}(w(x)) 
                              \right) \cdot  
                      \theta \left(t^{N-2} \hat{\rho}\left(\frac{t}{\ve}\right)\right) \mathrm{d}t \,\mathrm{d}\theta \dx\;.
\end{aligned}
\end{equation}
Now we will use a Taylor series of power 0 for $ t\mapsto \nabla w(x-t \theta)$ and of power 1 for $t \mapsto (\exp_{w(x-t \theta)})^{-1}(w(x))$ to rewrite \autoref{eq:reg3}.
We write 
\begin{equation}
 F(w;x,t,\theta) \defeq (\exp_{w(x-t \theta)})^{-1}(w(x)) \in \R^{M \times 1}
\end{equation}
and define 
\begin{equation}
 \dot{F}(w;x,\theta) \defeq \lim_{t \searrow 0} \frac{1}{t} \left((\exp_{w(x-t \theta)})^{-1}(w(x)) - 
 \underbrace{(\exp_{w(x)})^{-1}(w(x))}_{=0} 
 \right) \in \R^{M \times 1}.
\end{equation}
Note that because $(\exp_{w(x)})^{-1}(w(x))$ vanishes, $\dot{F}(w(x);\theta)$ is the leading order term of the expansion of 
$(\exp_{w(x-t \theta)})^{-1}(w(x))$ with respect to $t$.
Moreover, in the case that $\nabla w(x) \neq 0$ this is the leading order approximation of $\nabla w(x-t \theta)$. In summary we are calculating the leading order term of the expansion 
with respect to $t$.

Then from \autoref{eq:reg3} it follows that
\begin{equation}
\label{eq:reg3a}
\lim_{\ve \searrow 0}
\Reg_{[\dKo],\ve}(w) 
\approx
\lim_{\ve \searrow 0}
\underbrace{\frac{1}{\ve^N} \int\limits_0^\infty t^{N-1} \hat{\rho}\left(\frac{t}{\ve}\right) \mathrm{d}t}_{= \abs{\mathbb{S}^{N-1}}^{-1}}
 \int\limits_{\Omega} \int\limits_{\mathbb{S}^{N-1}} \left((\nabla w(x))^T \dot{F}(w;x,\theta) \right) \cdot \theta \;
  \mathrm{d}\theta \dx\;.
\end{equation}
The previous calculations show that the double integral simplifies to a double integral where the inner integration domain 
has one dimension less than the original integral. Under certain assumption the integration domain can be further simplified:

\begin{example}
 If $d_1(x,y)=\normM[x-y]$, $p_1=2$, then 
 \begin{equation*}
  \dot{F}(w;x,\theta) = \lim_{t \searrow 0} \frac{1}{t} \left(w(x) - w(x-t\theta)\right) = \nabla w(x)\theta \in \R^{M \times 1}.
 \end{equation*}
 Thus from \eqref{eq:reg3a} it follows that 
\begin{equation}
\label{eq:reg3b}
\lim_{\ve \searrow 0}
\Reg_{[\dKo],\ve}(w) 
\approx \int\limits_{\Omega} \underbrace{(\nabla w(x))^T \nabla w(x)}_{\norm{\nabla w(x)}^2_{\R^M}} \dx\;.
\end{equation}
This is exactly the identity derived in \citeauthor{BouBreMir01} \cite{BouBreMir01}.
\end{example}
From these considerations we can view $\lim_{\ve \searrow 0} \Reg_{[\dKo],\ve}$ as functionals, which generalize 
Sobolev and $\mathrm{BV}$ semi-norms to functions with values on manifolds.

\subsection{A conjecture on Sobolev semi-norms}
\label{ss:conjecture}
Starting point for this conjecture is \autoref{eq:d3}. We will write $\Omega,M$ and $p$ instead of $\Omega_1, M_1$ and $p_1$.
\begin{itemize}
 \item In the case $l=0$, $k=N$, $0<s<1$ and $\dKo(w(x), w(y))= \normM[w(x)-w(y)]$ 
       the functional $\Reg_{[\dKo]}$ from 
       \autoref{eq:d3} simplifies to the $p$-th 
       power of the Sobolev semi-norm and reads
       \begin{equation}\label{eq:d3a} 
        \int\limits_{\Omega\times \Omega} \frac{\normM[w(x)-w(y)]^{p}}{\normN[x-y]^{N+p s}} \dxy.
       \end{equation}
       For a recent survey on fractional Sobolev Spaces see \cite{DiNPalVal12}.
 \item On the other hand, when we choose $k=0$, $l=1$ and $\dKo(w(x), w(y))= \normM[w(x)-w(y)]$, then 
       $\Reg_{[\dKo]}$ from \autoref{eq:d3} reads
       (note $\rho=\rho_\ve$ by simplification of notation):
       \begin{equation}\label{eq:d3b} 
        \int\limits_{\Omega\times \Omega} \frac{\normM[w(x)-w(y)]^{p}}{\normN[x-y]^{p s}} 
        \rho_\ve(x-y) 
        \dxy.
    \end{equation}
 \item Therefore, in analogy to what we know for $s=1$ from \cite{BouBreMir01}, we conjecture that
        \begin{equation}\label{eq:d3c}
        \lim_{\ve \to 0} \int\limits_{\Omega\times \Omega} \frac{\normM[w(x)-w(y)]^{p}}{\normN[x-y]^{p s}} 
        \rho_\ve(x-y) 
          \dxy = C \int\limits_{\Omega\times \Omega} \frac{\normM[w(x)-w(y)]^{p}}{\normN[x-y]^{N+p s}}\dxy.
    \end{equation}
    The form \autoref{eq:d3c} is numerically preferable to the standard Sobolev semi-norm \autoref{eq:d3a}, because 
    $\rho=\rho_\ve$ and thus the integral kernel has compact support.
\end{itemize}

\section{Numerical Examples} \label{sec:Numerical_results}

In this section we present some numerical examples for denoising and inpainting of functions with values on the circle 
$\sphere$. Functions with values on a sphere have already been investigated very diligently (see for instance 
\cite{BouBreMir00b} out of series of publications of these authors). 
Therefore we review some of their results first. 

\subsection{$\sphere$-Valued Data}
\label{ss: spheredata}
Let $\emptyset \neq \Omega \subset \R$ or $\R^2$ be a bounded and simply connected open set with Lipschitz boundary. In \cite{BouBreMir00b}
the question was considered when $w \in \Wsp[\sphere]$ can be represented by some function $u \in \Wsp[\R]$ 
satisfying 
\begin{equation}\label{eq:id_sphere_w}
 \Phi(u) \defeq \e^{\i u} = w.
\end{equation}
That is, the function $u$ is a \emph{lifting} of $w$.

\begin{lemma}[\cite{BouBreMir00b}] \label{lem: lifting}
\begin{itemize}
 \item Let $\Omega \subset \R$, $0 < s < \infty$, $1 < p < \infty$. Then for all $w \in \Wsp[\sphere]$ there exists 
       $u \in \Wsp[\R]$ satisfying \autoref{eq:id_sphere_w}.
 \item Let $\Omega \subset \R^N$, $N \geq 2$, $0 < s < 1$, $1 < p < \infty$. Moreover, let 
       $sp < 1$ or $sp \geq N$, then for all $w \in \Wsp[\sphere]$ there exists 
       $u \in \Wsp[\R]$ satisfying \autoref{eq:id_sphere_w}. 
       
       If  $sp \in [1,N)$, then there exist functions $w \in \Wsp[\sphere]$ such that \autoref{eq:id_sphere_w} does not hold with any function
       $u \in \Wsp[\R]$. 
\end{itemize}

\end{lemma}
For 
\begin{equation}
 \label{eq:arccos}
 \dS(a,b) \defeq \arccos(a^T b) \, , \quad a, b \in \sphere,
\end{equation}
we consider the functional (note that by simplification of notation below $\rho=\rho_\ve$ denotes a mollifier)
\begin{equation}
\label{eq:considered}
\Reg_{[\dS]}(w) = \int\limits_{\Omega\times \Omega} \frac{\dS^p(w(x), w(y))}{\normN[x-y]^{k+ps}} \rho^l(x-y) \dxy,
\end{equation}
on $w \in \Wsp[\sphere]$, 
in accordance to \autoref{eq:d3}.

Writing $w = \Phi(u)$ as in \autoref{eq:id_sphere_w}
we get the 
lifted functional
\begin{equation}
\label{eq:sobolev_alternative}
\Reg_{[\dS]}^{\Phi}(u) \defeq
 \int\limits_{\Omega\times \Omega} \frac{\dS^p(\Phi(u)(x), \Phi(u)(y))}{\normN[x-y]^{k+ps}} \rho^l(x-y) \dxy,
\end{equation}
over the space $\Wsp[\R]$.
\begin{remark}
\begin{itemize}
 \item We note that in the case $k=0$, $s=1$ and $l=1$ these integrals correspond with the ones considered 
 in \citeauthor{BouBreMir01} \cite{BouBreMir01} for functions with values on $\sphere$.
 \item 
 If we choose $k=N$, $s=1$ and $l=0$, then this corresponds with Sobolev semi-norms on manifolds. 
 \item Let $\ve > 0$ fixed (that is, we consider neither a standard Sobolev regularization nor the limiting 
       case $\ve \to 0$ as in \cite{BouBreMir01}). In this case we have proven coercivity of the functional 
       $\F: \Wsp[\sphere] \rarr [0,\infty), \ 0<s<1,$ only with the following regularization functional, 
       cf. \autoref{ex:coercive} and \autoref{ex:in}:   
       $$\int\limits_{\Omega\times \Omega} \frac{\dS^p(w(x),w(y))}{\normN[x-y]^{N + p s}} \rho_\ve(x-y) \dxy.$$
 \end{itemize}
\end{remark}

We summarize a few results: The first lemma follows from elementary calculations:
\begin{lemma}
\label{le:1}
 $\dS$ and $\d_{\R^2}\big|_{\sphere \times \sphere}$ are equivalent.
\end{lemma}

\begin{lemma} \label{lem:2}
	Let $u \in \Wsp[\R]$. Then $\Phi(u) \in \Wsp[\sphere]$.
\end{lemma}

\begin{proof}
	This follows directly from the inequality $\|\e^{ia}-\e^{ib}\| \leq \|a-b\|$ for all $a,b \in \R$.
\end{proof}
Below we show that $\Reg_{[\dS]}^{\Phi}$ is finite on $\Wsp[\R]$.
\begin{lemma} \label{lem: liftedRegularizer}
	$\Reg_{[\dS]}^{\Phi}$ maps $\Wsp[\R]$ into $[0,\infty)$ (i.e. does not attain the value $+\infty$).
\end{lemma}
\begin{proof}
	Let $u \in \Wsp[\R]$. Then by Lemma \ref{lem:2} we have that $\Phi(u) \in \Wsp[\sphere]$. Therefore, from Lemma \ref{le:1} and \autoref{pr:ExprIsOp} \autoref{itm: ExpIsOp} it follws that $\Reg_{[\dS]}(\Phi(u))< \infty$. Hence, by definition, $\Reg_{[\dS]}^{\Phi}(u) < \infty$.  
\end{proof}

\subsection{Setting of numerical examples}
In all numerical examples presented we use a simplified setting with
\begin{equation*}
 M_1 = M_2 \eqdef M,\;K_1 = K_2 \eqdef \sphere,\;p_1 = p_2 \eqdef p,\;k = N,\;l = 1,
\end{equation*}
$\Omega_1 = \Omega_2 \eqdef \Omega$ when considering image denoising,
$\Omega_1 = \Omega$, $\Omega_2 = \Omega \setminus D$ when considering image inpainting,
and
\begin{equation*}
W(\Omega,\sphere) = \Wsp[\sphere].
\end{equation*}
As particular mollifier we use $\rho_\ve$ (see \autoref{ex:mol}), which is defined via the one-dimensional 
normal-distribution $ \hat{\rho}(x) = \frac{1}{\sqrt{\pi}} \e^{-x^2}.$

\subsection*{Regularization functionals}
Let $\Reg_{[\dS]}$ and $\Reg_{[\dS]}^{\Phi}$ be as defined in \autoref{eq:considered} and \autoref{eq:sobolev_alternative}, respectively.
In what follows we consider the following regularization functional 
\begin{equation}
\label{eq:reg_numerics}
\F<v^\delta><\alpha>[\dS](w) \defeq \int\limits_\Omega \dS^p(\op[w](x), v^\delta(x)) \dx + \alpha \Reg_{[\dS]}(w),
\end{equation}
on $\Wsp[\sphere]$ and the lifted variant
\begin{equation} \label{eq: functionalAlternative}
\FT<v^\delta><\alpha>[\dS](u) \defeq 
\int\limits_\Omega \dS^p(\op[\Phi(u)](x), v^\delta(x)) \dx + \alpha \Reg_{[\dS]}^{\Phi}(u) 
\end{equation}
over the space $\Wsp[\R]$ (as in \autoref{ss: spheredata}), where $\Phi$ is defined as in \eqref{eq:id_sphere_w}.
Note that $\FT = \F \circ \Phi$.

\begin{lemma}\label{lem: liftedFunctional}
Let $\emptyset \neq \Omega \subset \R$ or $\R^2$ be a bounded and simply connected open set with Lipschitz boundary. 
Let $1 < p < \infty$ and $s \in (0,1)$. If $N=2$ assume that $sp < 1$ or $sp \geq 2$. Moreover, let \autoref{as:Setting} and \autoref{ass:2} be satisfied. Then the mapping 
$\FT<v^\delta><\alpha>[\dS]: W^{s,p}(\Omega, \R) \rarr [0,\infty)$ attains a minimizer.
\end{lemma}
\begin{proof}
	Let $u \in \Wsp[\R]$. Then by Lemma \ref{lem:2} we have that $w \defeq \Phi(u) \in \Wsp[\sphere]$.
	As arguing as in the proof of Lemma \ref{lem: liftedRegularizer} we see that $\FT<v^\delta><\alpha>[\dS](u) < \infty$. \\
	Since we assume that \autoref{as:Setting} is satisfied we get that $\F<v^\delta><\alpha>[\dS](w)$ attains a minimizer $w^* \in \Wsp[\sphere]$.
	It follows from \autoref{lem: lifting} that there exists a function $u^* \in W^{s,p}(\Omega, \R)$ that can be lifted to $w^*$, i.e. $w^* = \Phi(u^*)$.
	Then $u^*$ is a minimizer of \eqref{eq: functionalAlternative} 
	by definition of $\FT$ and $\Phi$.
\end{proof}

\subsection{Numerical minimization}
In our concrete examples we will consider two different operators $\op$. 
For numerical minimization we consider the functional from \autoref{eq: functionalAlternative} 
in a discretized setting. For this purpose we approximate the functions $u \in W^{s, p}(\Omega,\R)$, $0<s<1,1<p<\infty$ by quadratic B-Spline 
functions and optimize with respect to the coefficients. 
We remark that this approximation 
is continuous and thus that sharp edges correspond to very steep slopes. \\ 
The noisy data $u^\delta$ is obtained by adding Gaussian white noise with variance $\sigma^2$ to the approximation 
or the discretized approximation of $u$. 

We apply a simple Gradient Descent scheme with fixed step length implemented in $\mathrm{MATLAB}$.

\subsection{Denoising of $\sphere$-valued functions - The InSAR problem} \label{ss: denoising}
In this case the operator $\op: \Wsp[\sphere] \rarr \Lp[\sphere]$ is the inclusion operator. It is norm-coercive in the sense of \autoref{eq:coercive_F} and hence \autoref{as:Setting} is fulfilled. 
For 
$\emptyset \neq \Omega \subset \R$ or $\R^2$ a bounded and simply connected open set, $1 < p < \infty$ and $s \in (0,1)$ such that additionally $sp < 1$ or $sp \geq 2$ if $N=2$ we can apply \autoref{lem: liftedFunctional} which ensures that the lifted functional 
$\FT<u^\delta><\alpha>[\dS]: \Wsp[\R] \rarr [0,\infty)$ 
attains a minimizer $u \in W^{s, p}(\Omega,\R)$. 

In the examples we will just consider the continuous approximation again denoted by $u$.

\subsection*{One dimensional test case}
Let $\Omega = (0,1)$ and consider the signal 
$u:\Omega \rarr [0,2\pi)$ representing the angle of a cyclic signal. \\
For the discrete approximation shown in \autoref{sfig:signal1-a} the domain $\Omega$ is sampled equally at 100 points.    
$u$ is affected by an additive white Gaussian noise with $\sigma = 0.1$ to obtain the noisy signal which is colored in blue in \autoref{sfig:signal1-a}.
 
In this experiment we show the influence of the parameters $s$ and $p$.
In all cases the choice of the regularization parameter $\alpha$ is 0.19 and $\varepsilon = 0.01$.\\
The red signal in \autoref{sfig:signal1-b} is obtained by choosing 
$s = 0.1$ and $p = 1.1$. 
We see that the periodicity of the signal is handled correctly and that there is nearly no staircasing.
In \autoref{sfig:signal1-c} the parameter $s$ is changed from $0.1$ to $0.6$. The value of the parameter $p$ stays fixed. 
Increasing of $s$ leads the signal to be more smooth. We can observe an even stronger similar effect when increasing $p$ 
(here from $1.1$ to $2$) and letting $s$ fixed, see \autoref{sfig:signal1-d}. This fits the expectation since $s$ only 
appears once in the denominator of the regularizer. At a jump increasing of $s$ leads thus to an increasing of the 
regularization term. The parameter $p$ appears twice in the regularizer. Huge jumps are hence weighted even more.

\begin{figure}[!h]
	\centering
	\begin{subfigure}[h]{0.35\linewidth}
		\includegraphics[width=1\linewidth]{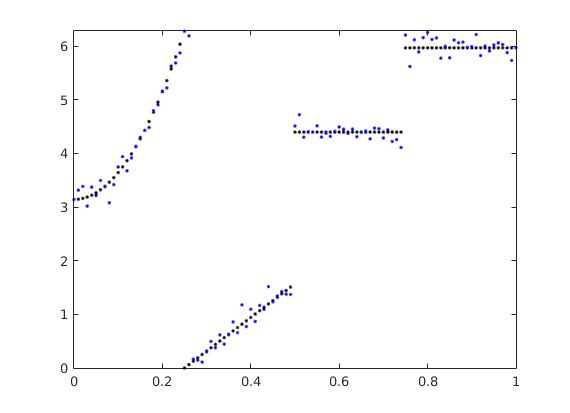}
		\caption{Original and noisy data}
		\label{sfig:signal1-a}
	\end{subfigure}
    \begin{subfigure}[h]{0.35\linewidth}
    	\includegraphics[width=1\linewidth]{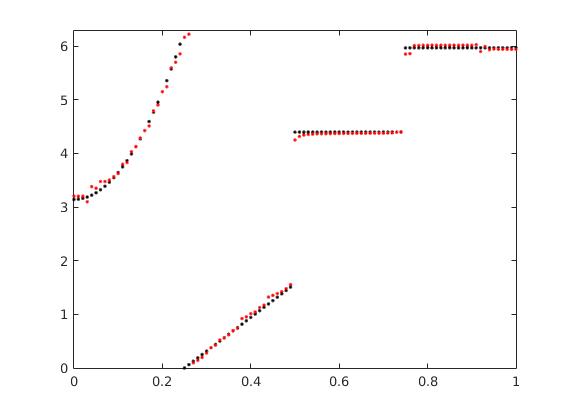}
    	\caption{Denoised data}
    	\label{sfig:signal1-b}
    \end{subfigure}	

	\begin{subfigure}[h]{0.35\linewidth}
		\includegraphics[width=1\linewidth]{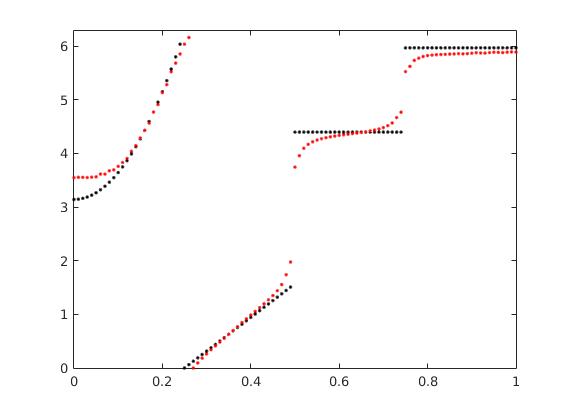}
		\caption{Increasing of $s$}
		\label{sfig:signal1-c}
	\end{subfigure}
    \begin{subfigure}[h]{0.35\linewidth}
    	\includegraphics[width=1\linewidth]{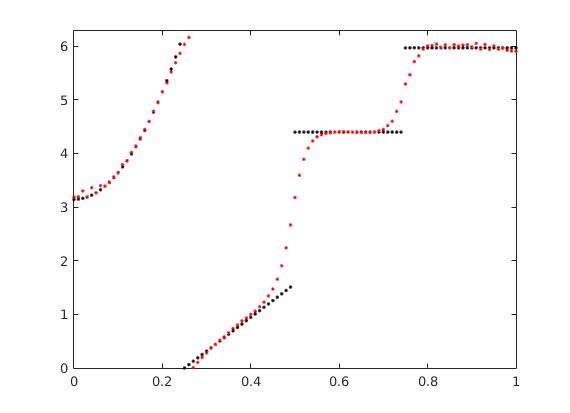}
    	\caption{Increasing of $p$}
    	\label{sfig:signal1-d}
    \end{subfigure}
	
\caption{Function on $\sphere$ represented in $[0,2\pi)$: Left to right, top to bottom: Original data (black) and noisy data (blue) 
with 100 data points. Denoised data (red) where we chose $s=0.1, p=1.1, \alpha = 0.19$. Denoised data with 
$s=0.6, p=1.1, \alpha = 0.19$ resp. $s=0.1, p=2, \alpha=0.19$. }
\label{fig:signal1}
\end{figure}

In \autoref{sfig:signal2-a} we considered a simple signal with a single huge jump. Again it is described by the angular value. 
We proceeded as above to obtain the approximated discrete original data (black) and noisy signal with $\sigma = 0.1$ (blue). We 
chose again $\varepsilon = 0.01$. \\ 
As we have seen above increasing of $s$ leads to a more smooth signal. 
This effect can be compensated by choosing a rather small value of $p$, i.e. $p \approx 1$. In \autoref{sfig:signal2-b} the value of $s$ is $0.9$. We see that it is still possible to reconstruct jumps by choosing e.g. $p=1.01$. \\
Moreover, we have seen that increasing of $p$ leads to an even more smooth signal. In \autoref{sfig:signal2-c} we choose a quite large value of $p$,  
$p=2$ and a 
rather small value of $s$, $s = 0.001$. Even for this very simple signal is was not possible to get sharp edges. This is due to the fact that the parameter $p$ (but not $s$) additionally weights the height of jumps in the regularizing term.

\begin{figure}[!h]
	\centering
	\begin{subfigure}[h]{0.3\linewidth}
	\includegraphics[width=1\linewidth]{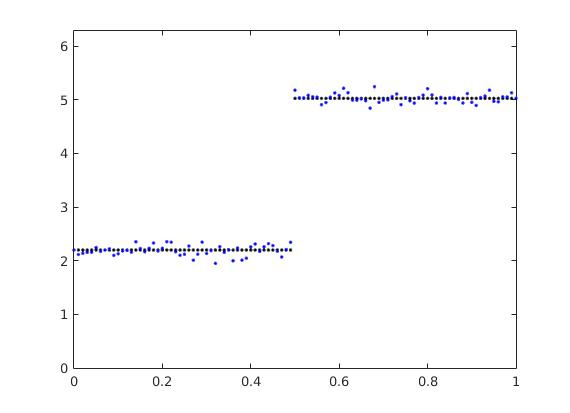}
	\caption{Original and noisy data}
	\label{sfig:signal2-a}
	\end{subfigure}
	\begin{subfigure}[h]{0.3\linewidth}
		\includegraphics[width=1\linewidth]{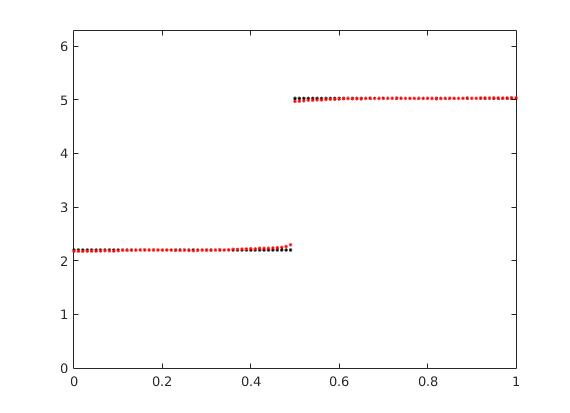}
		\caption{$s = 0.9, \ p = 1.01$}
		\label{sfig:signal2-b}
	\end{subfigure}
    \begin{subfigure}[h]{0.3\linewidth}
    	\includegraphics[width=1\linewidth]{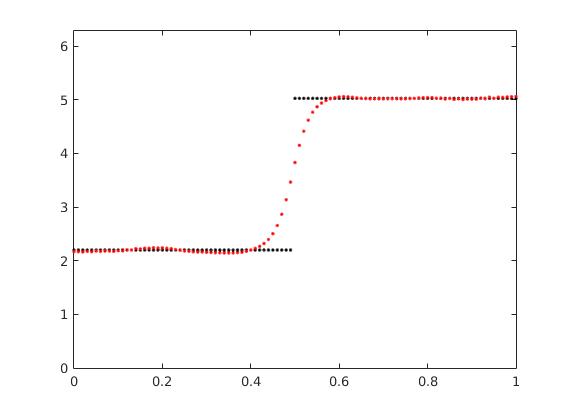}
    	\caption{$s = 0.001, \ p = 2$}
    	\label{sfig:signal2-c}
    \end{subfigure}
    \caption{Left to right: Original data (black) and noisy data (blue) sampled at 100 data points. Denoised data (red) where we chose $s=0.9, p=1.01, \alpha = 0.03$. Denoised data with $s=0.001, p=2, \alpha = 0.9$.}
	\label{fig:signal2}
\end{figure}

\subsection*{Denoising of a $\sphere$-Valued Image}
Our next example concerned a two-dimensional $\sphere$-valued image represented by the corresponding angular values.
We remark that in this case where $N=2$ the existence of such a representation is always guaranteed in the cases where 
$sp < 1$ or $sp \geq 2$, see \autoref{lem: lifting}. 

The domain $\Omega$ is sampled into $60 \times 60$ data points and can be considered as discrete grid, 
$\{1, \dots,60\} \times \{1, \dots,60\} $. 
The B-Spline approximation evaluated at that grid is given by
\begin{equation*}
u(i,j) = u(i,0) \defeq 4\pi \frac{i}{60} \bmod 2\pi, \quad i,j \in \{1, \dots,60\}.
\end{equation*}

\begin{figure} 
	\centering
	{\includegraphics[width=6cm]{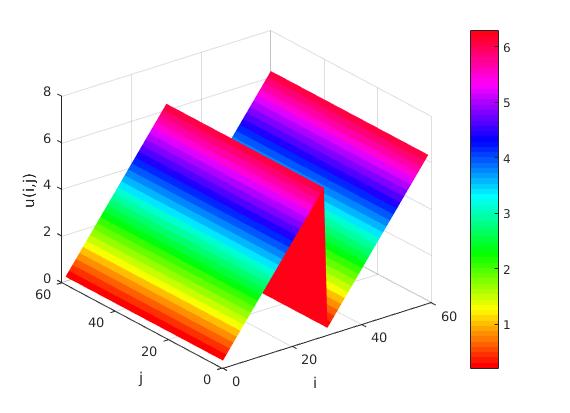}}
	\caption{The function $u$ evaluated on the discrete grid.}
	\label{fig: Rainbow_function}
\end{figure}
The function $u$ is shown in \autoref{fig: Rainbow_function}. We used the $\mathrm{hsv}$ colormap 
provided in $\mathrm{MATLAB}$ transferred to the interval $[0, 2\pi]$. 

This experiment shows the difference of our regularizer respecting the periodicity of the data in contrast to the classical 
Total Variation regularizer. The classical TV-minimization is solved using a fixed point iteration (\cite{LoeMag}); for the method see also \cite{VogOma96}.

In \autoref{sfig:rainbow-a} the function $u$ can be seen from the top, i.e. the axes correspond to the 
$i$ resp. $j$ axis in \autoref{fig: Rainbow_function}.
The noisy data is obtained by adding white Gaussian noise with $\sigma = \sqrt{0.001}$ using the built-in function 
$\mathtt{imnoise}$ in $\mathrm{MATLAB}$. It is shown in \autoref{sfig:rainbow-b}.
We choose as parameters $s=0.9, \ p=1.1, \ \alpha = 1,$ and $\varepsilon = 0.01$.  
We observe significant noise reduction in both cases. However, only in \autoref{sfig:rainbow-d} the color transitions are handled 
correctly. This is due to the fact, that our regularizer respects the periodicity, i.e. for the functional there is no jump in 
\autoref{fig: Rainbow_function} since 0 and $2\pi$ are identified. Using the classical TV regularizer the values 0 and $2\pi$ are not 
identified and have a distance of $2\pi$. Hence, in the TV-denoised image there is a sharp edge in the middle of the image, see 
\autoref{sfig:rainbow-c}. \\

\begin{figure}[!h]
	\centering
	\begin{subfigure}[h]{0.35\linewidth}
		\includegraphics[width=1\linewidth]{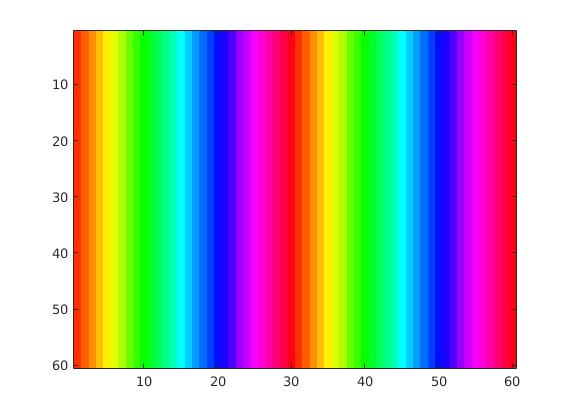}
		\caption{Original data}
		\label{sfig:rainbow-a}
	\end{subfigure}
	\begin{subfigure}[h]{0.35\linewidth}
		\includegraphics[width=1\linewidth]{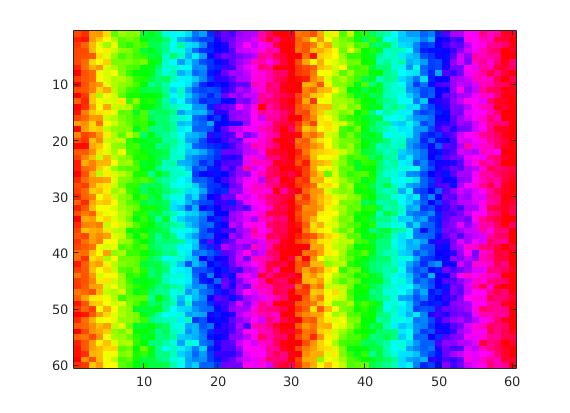}
		\caption{Noisy data}
		\label{sfig:rainbow-b}
	\end{subfigure}	
	
	\begin{subfigure}[h]{0.35\linewidth}
		\includegraphics[width=1\linewidth]{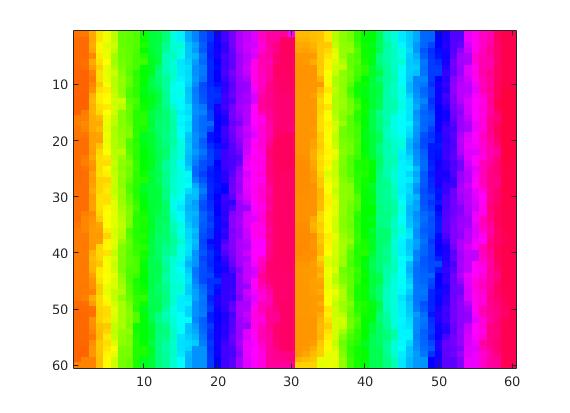}
		\caption{TV-denoised data}
		\label{sfig:rainbow-c}
	\end{subfigure}
	\begin{subfigure}[h]{0.35\linewidth}
		\includegraphics[width=1\linewidth]{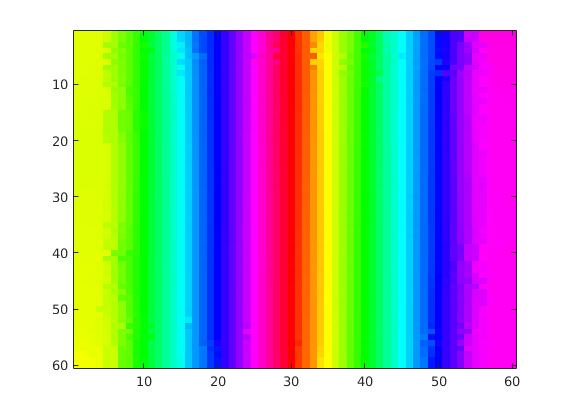}
		\caption{Denoised data}
		\label{sfig:rainbow-d}
	\end{subfigure}
	
	\caption{Left to right, top to bottom: Original and noisy data of an $60 \times 60$ image. TV-denoised data using a fixed point iteration method. Denoised data where we chose $s=0.9, p=1.1, \alpha = 1$, 400 steps. }
	\label{fig:rainbow}
\end{figure}

\subsection*{Hue Denoising}

The $\mathrm{HSV}$ color space is shorthand for Hue, Saturation, Value (of brightness). 
The hue value of a color image is $\sphere$-valued, while saturation and value of brightness are real-valued. 
Representing colors in this space better match the human perception than representing colors in 
the RGB space. 

In \autoref{sfig:fruits-a} we see a part of size $70 \times 70$ of the RGB image ``fruits''  
(\url{https://homepages.cae.wisc.edu/~ece533/images/}).
     
The corresponding hue data is shown in \autoref{sfig:fruits-b}, where we used again the colormap hsv, cf. \autoref{fig: Rainbow_function}. Each pixel-value lies, after transformation, in the interval $[0, 2\pi)$ and represents the angular value. Gaussian white noise with $\sigma = \sqrt{0.001}$ is added to obtain a noisy image, see \autoref{sfig:fruits-c}.\\
To obtain the denoised image \autoref{sfig:fruits-d} we again used the same fixed point iteration, cf. \cite{LoeMag}, as before. 

We see that the denoised image suffers from artifacts due to the non-consideration of periodicity. The pixel-values in the middle of the apple (the red object in the original image) are close to $2\pi$ while those close to the border are nearly 0, meaning they have a distance of around $2\pi$. \\
We use this TV-denoised image as starting image to perform the minimization of our energy functional. As parameters we choose 
$s = 0.49, \ p = 2, \ \alpha = 2, \ \varepsilon = 0.006$. 

Since the cyclic structure is respected the disturbing artifacts in image \autoref{sfig:fruits-d} are removed correctly. The edges are smoothed due to the high value of $p$, see \autoref{sfig:fruits-e}.

\begin{figure}[!h]
	\centering
	\begin{subfigure}[h]{0.3\linewidth}
		\includegraphics[width=1\linewidth]{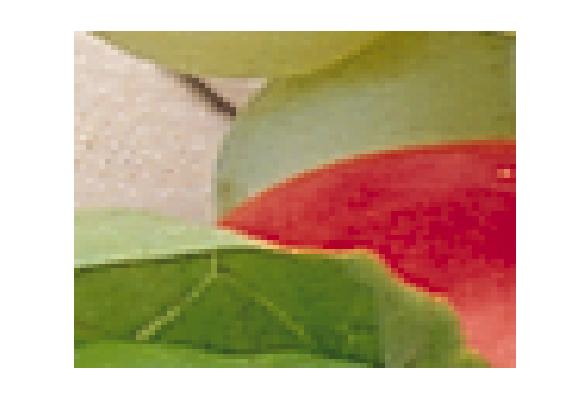}
		\caption{Original RGB image \newline \qquad \newline}
		\label{sfig:fruits-a}
	\end{subfigure}
    \hspace{1em}
	\begin{subfigure}[h]{0.3\linewidth}
		\includegraphics[width=1\linewidth]{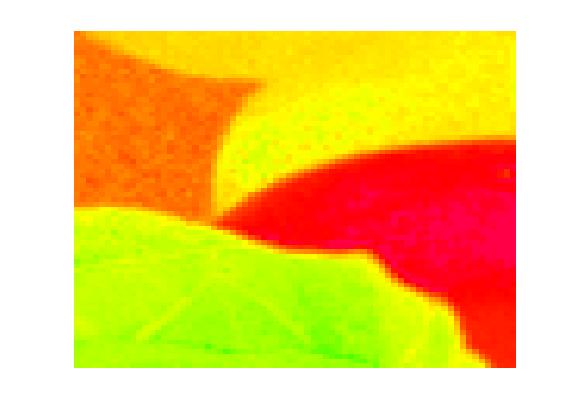}
		\caption{Hue component represented in color, which represent function values on $\sphere$}
		\label{sfig:fruits-b}
	\end{subfigure}	
    \hspace{1em}	
	\begin{subfigure}[h]{0.3\linewidth}
		\includegraphics[width=1\linewidth]{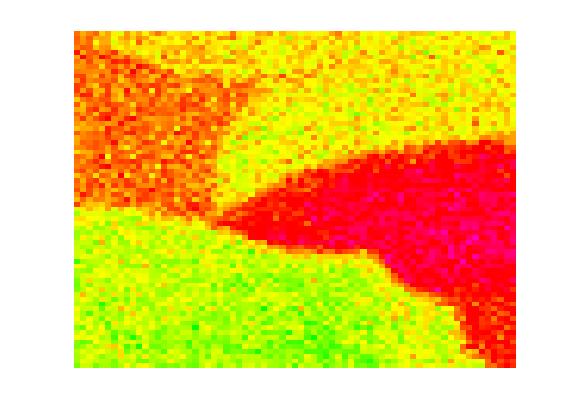}
		\caption{Noisy hue value - again representing function values on $\sphere$ \newline}
		\label{sfig:fruits-c}
	\end{subfigure}

	\begin{subfigure}[h]{0.3\linewidth}
		\includegraphics[width=1\linewidth]{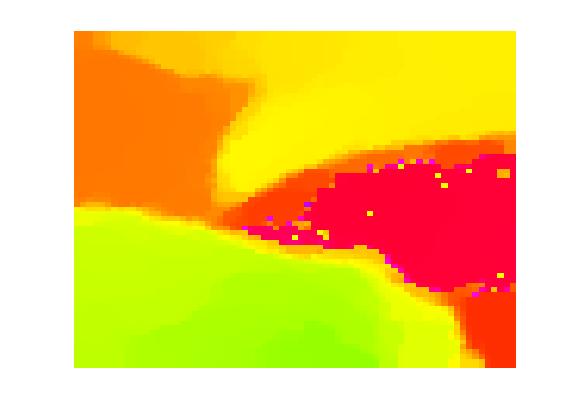}
		\caption{TV-denoised data}
		\label{sfig:fruits-d}
	\end{subfigure}
    \hspace{1em}
    \begin{subfigure}[h]{0.3\linewidth}
    	\includegraphics[width=1\linewidth]{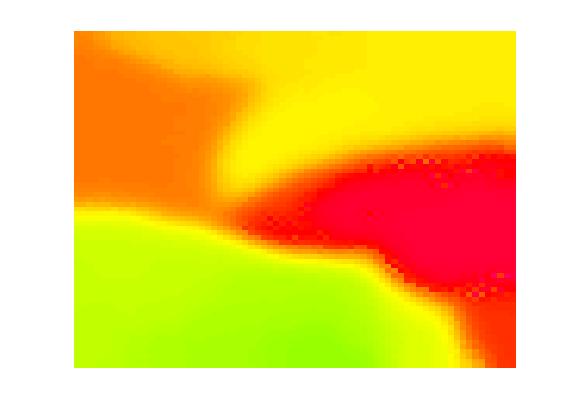}
    	\caption{Denoised data}
    	\label{sfig:fruits-e}
    \end{subfigure}
	
	\caption{Left to right, top to bottom: Original RGB image and its Hue component. Noisy Hue data with $\sigma^2 = 0.001$. TV minimization is done using an iterative approach. It is serving as starting point for the GD minimization. Denoised data with $s=0.49, p=2, \alpha = 2$, 500 steps. }
	\label{fig:fruits}
\end{figure}

\subsection{$\sphere$-Valued Image Inpainting}
In this case the operator $\op: \Wsp[\sphere] \rarr \Lp[\sphere]$ is the inpainting operator, i.e. 
  \begin{equation*}
    \op(w) = \chi_{\Omega \backslash D} (w),
  \end{equation*}  
where $D \subseteq \Omega$ is the area to be inpainted. 

We consider the functional 
\begin{equation*}
\F<v^\delta><\alpha>[\dS](w) \defeq \int\limits_{\Omega \setminus D} \dS^p(w(x), v^\delta(x)) \dx + 
\alpha \int\limits_{\Omega\times \Omega} \frac{\dS^p(w(x), w(y))}{\norm{x-y}_{\R^2}^{2+ps}} \rho_{\ve}(x-y) \dxy,
\end{equation*}
on $\Wsp[\sphere]$.

According to \autoref{ex:in} the functional $\F$ is coercive and 
\autoref{as:Setting} is satisfied.
For $\emptyset \neq \Omega \subset \R$ or $\R^2$ a bounded and simply connected open set, $1 < p < \infty$ and $s \in (0,1)$ such that additionally $sp < 1$ or $sp \geq 2$ if $N=2$ \autoref{lem: liftedFunctional} applies which ensures that there exists a minimizer $u \in W^{s, p}(\Omega,\R)$ of the lifted functional 
$\FT<u^\delta><\alpha>[\dS]: \Wsp[\R] \rarr [0,\infty)$ 
$u \in W^{s, p}(\Omega,\R)$

\subsection*{Inpainting of a $\sphere$-Valued Image}

As a first inpainting test-example we consider two $\sphere$-valued images of size $28 \times 28$, 
see \autoref{fig:blocks}, represented by its angular values. 
In both cases the ground truth can be seen in \autoref{sfig:blocks-a} and \autoref{sfig:blocks2-a}. 
We added Gaussian white noise with $\sigma = \sqrt{0.001}$ using the $\mathrm{MATLAB}$ build-in function 
$\mathtt{imnoise}$. The noisy images can be seen in \autoref{sfig:blocks-b} and \autoref{sfig:blocks2-b}. 
The region $D$ consists of the nine red squares in \autoref{sfig:blocks-c} and \autoref{sfig:blocks2-c}.

The reconstructed data are shown in \autoref{sfig:blocks-d} and \autoref{sfig:blocks2-d}. \\
For the two-colored image we used as parameters $\alpha = s = 0.3$, $p = 1.01$ and $\varepsilon = 0.05$.
We see that the reconstructed edge appears sharp. 
The unknown squares, which are completely surrounded by one color are inpainted perfectly.
The blue and green color changed slightly.

As parameters for the three-colored image we used $\alpha = s = 0.4$, $p=1.01$ and $\varepsilon = 0.05$. 
Here again the unknown regions lying entirely in one color are inpainted perfectly. The edges are preserved. 
Just the corner in the middle of the image is slightly smoothed. \\
In \autoref{sfig:blocks-e} and \autoref{sfig:blocks2-e} the TV-reconstructed data  
is shown. The underlying algorithm (\cite{Get}) uses the split Bregman method (see \cite{GolSta09}).

In \autoref{sfig:blocks-e} the edge is not completely sharp. There are some lighter parts on the blue side. 
This can be caused by the fact that the unknown domain in this area is not exactly symmetric with respect to the edge. 
This is also the case in \autoref{sfig:blocks2-e} where we observe the same effect. Unknown squares lying entirely 
in one color are perfectly inpainted.

\begin{figure}[!h]
	\centering
	\begin{subfigure}[h]{0.3\linewidth}
		\includegraphics[width=1\linewidth]{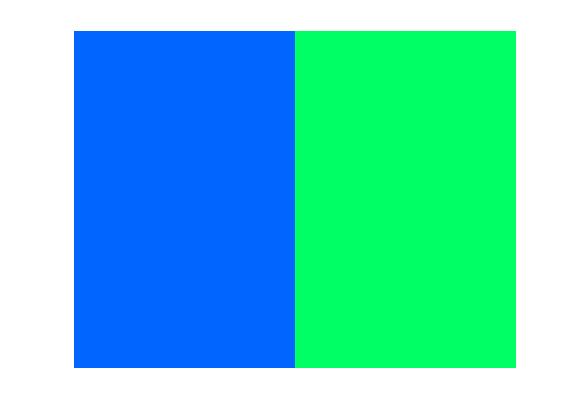}
		\caption{Original image}
		\label{sfig:blocks-a}
	\end{subfigure}
	\begin{subfigure}[h]{0.3\linewidth}
		\includegraphics[width=1\linewidth]{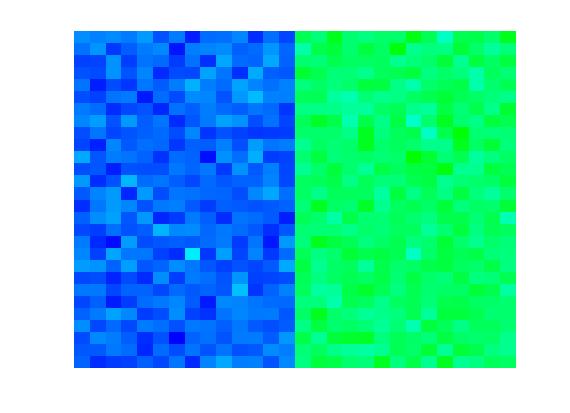}
		\caption{Noisy image}
		\label{sfig:blocks-b}
	\end{subfigure}		
	\begin{subfigure}[h]{0.3\linewidth}
		\includegraphics[width=1\linewidth]{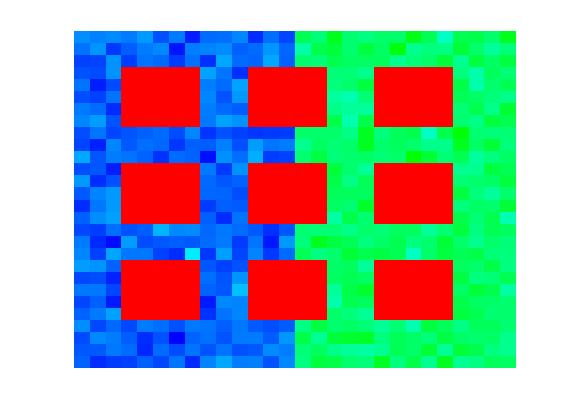}
		\caption{Noisy masked image}
		\label{sfig:blocks-c}
	\end{subfigure}
	
	\begin{subfigure}[h]{0.3\linewidth}
		\includegraphics[width=1\linewidth]{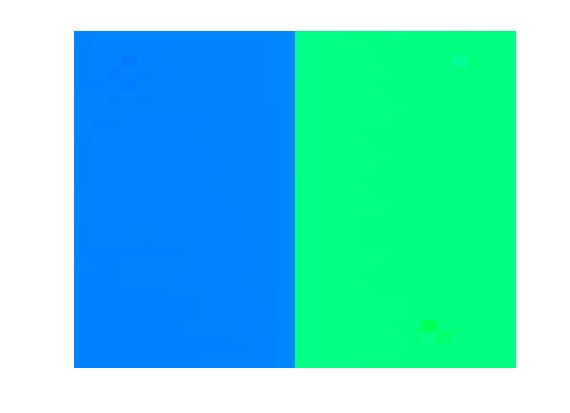}
		\caption{Reconstructed image}
		\label{sfig:blocks-d}
	\end{subfigure}
	\begin{subfigure}[h]{0.3\linewidth}
		\includegraphics[width=1\linewidth]{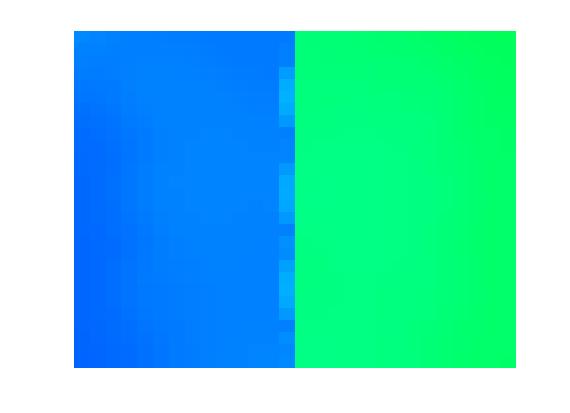}
		\caption{TV-reconstructed image}
		\label{sfig:blocks-e}
	\end{subfigure}

    \begin{subfigure}[h]{0.3\linewidth}
    	\includegraphics[width=1\linewidth]{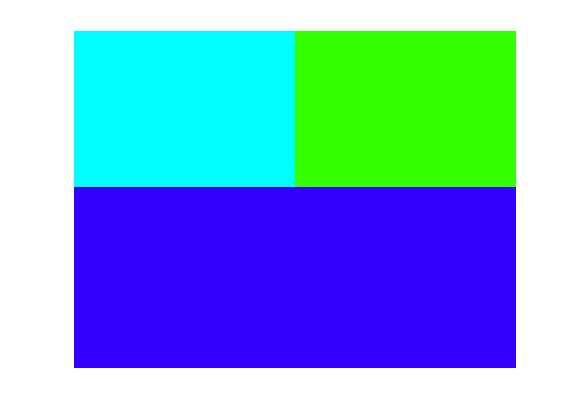}
    	\caption{Original image}
    	\label{sfig:blocks2-a}
    \end{subfigure}
    \begin{subfigure}[h]{0.3\linewidth}
    	\includegraphics[width=1\linewidth]{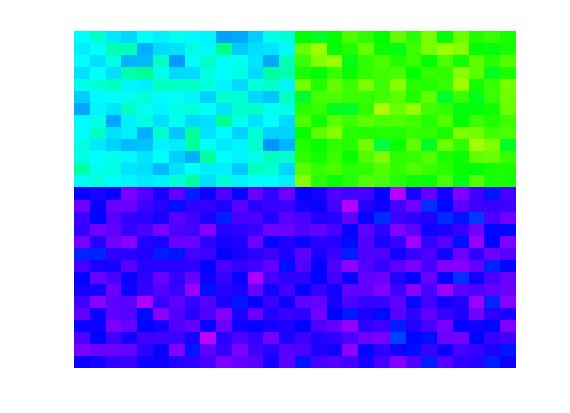}
    	\caption{Noisy image}
    	\label{sfig:blocks2-b}
    \end{subfigure}		
    \begin{subfigure}[h]{0.3\linewidth}
    	\includegraphics[width=1\linewidth]{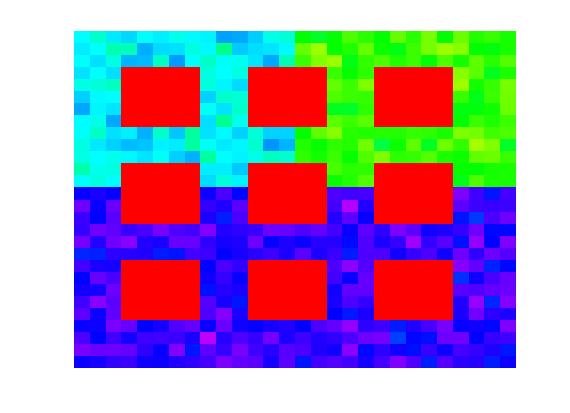}
    	\caption{Noisy masked image}
    	\label{sfig:blocks2-c}
    \end{subfigure}
    
    \begin{subfigure}[h]{0.3\linewidth}
    	\includegraphics[width=1\linewidth]{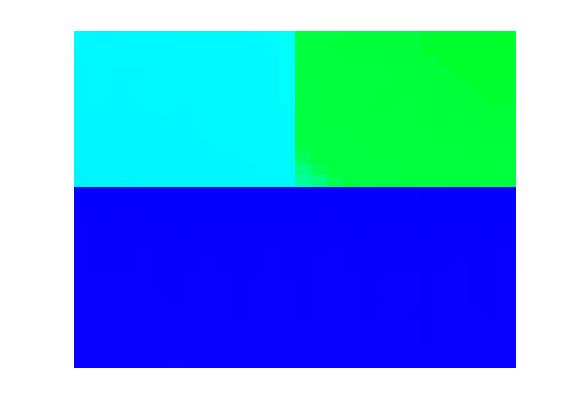}
    	\caption{Reconstructed image}
    	\label{sfig:blocks2-d}
    \end{subfigure}
    \begin{subfigure}[h]{0.3\linewidth}
    	\includegraphics[width=1\linewidth]{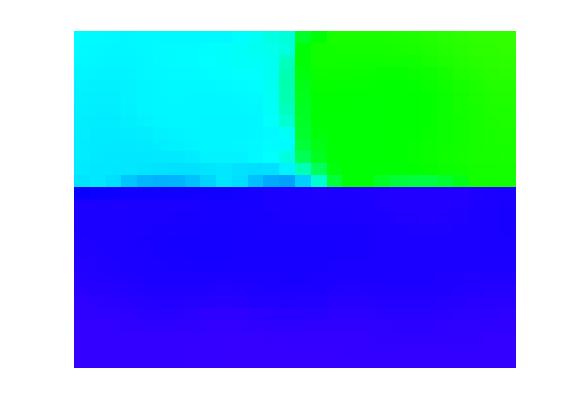}
    	\caption{TV-reconstructed image}
    	\label{sfig:blocks2-e}
    \end{subfigure}
	
	\caption{Left to right. Top to bottom: Original image and the noisy data with $\sigma^2 = 0.001$. Noisy image with masking filter and denoised data with $s=0.3, p=1.01, \alpha = 0.3$, 6000 steps. TV denoised data.\\
     Original image and the noisy data with $\sigma^2 = 0.001$. Noisy image with masking filter and denoised data with $s=0.4, p=1.01, \alpha = 0.4$, 10000 steps. TV denoised image.}
	\label{fig:blocks}
\end{figure}

\subsection*{Hue Inpainting}

As a last example we consider again the Hue-component of the image ``fruits'', see \autoref{sfig:fruits2-a}. 
The unknown region $D$ is the string $\mathit{01.01}$ which is shown in \autoref{sfig:fruits2-b}. 
As parameters we choose $p=1.1$, $s=0.1$, $\alpha= 2$ and $\ve = 0.006$. We get the reconstructed image shown in 
\autoref{sfig:fruits2-c}. The edges are preserved and the unknown area is restored quite well. This can be also observed in the 
TV reconstructed image, \autoref{sfig:fruits2-d}, using again the split Bregman method as before, cf. \cite{Get}.

\begin{figure}[!h]
	\centering
	\begin{subfigure}[h]{0.35\linewidth}
		\includegraphics[width=1\linewidth]{Images/Fruits_Orig.jpg}
		\caption{Hue component}
		\label{sfig:fruits2-a}
	\end{subfigure}		
	\begin{subfigure}[h]{0.35\linewidth}
		\includegraphics[width=1\linewidth]{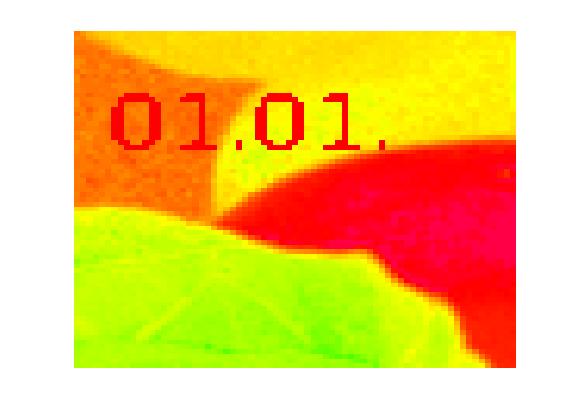}
		\caption{Image with masked region}
		\label{sfig:fruits2-b}
	\end{subfigure}
	
	\begin{subfigure}[h]{0.35\linewidth}
		\includegraphics[width=1\linewidth]{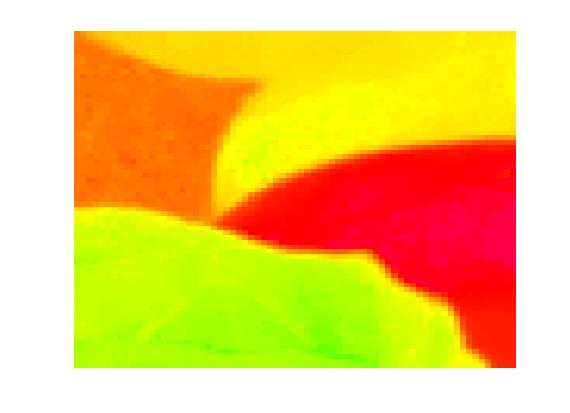}
		\caption{TV-reconstructed image}
		\label{sfig:fruits2-c}
	\end{subfigure}
	\begin{subfigure}[h]{0.35\linewidth}
		\includegraphics[width=1\linewidth]{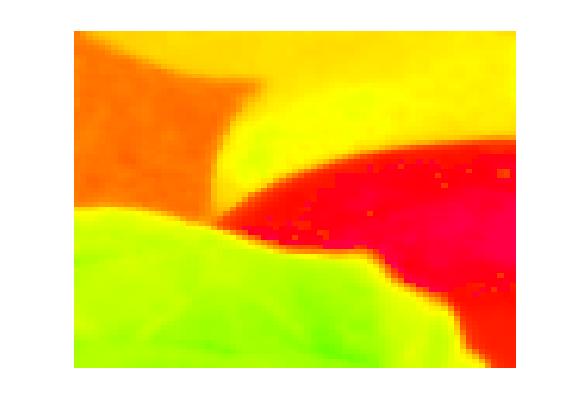}
		\caption{Reconstructed image}
		\label{sfig:fruits2-d}
	\end{subfigure}
	
	\caption{Left to right, top to bottom: Original image and image with masked region. Reconstructed image with parameters $p=1.1, \ 
	s=0.1, \ \alpha= 2$ and $\ve = 0.006$, 2000 steps. TV-reconstructed image.}
	\label{fig:fruits2}
\end{figure}

\subsection{Conclusion}
In this paper we developed a functional for regularization of functions with values in a set of vectors. The regularization functional  
is a derivative-free, nonlocal term, which is based on a characterization of Sobolev spaces of 
\emph{intensity data} derived by Bourgain, Brézis, Mironescu \& Dávila. Our objective has been 
to extend their double integral functionals in a natural way to functions with values in a set of vectors, in particular functions with 
values on an embedded manifold. These new integral representations are used for regularization on a subset 
of the (fractional) Sobolev space $W^{s,p}(\Omega, \R^M)$ and the space $BV(\Omega, \R^M)$, respectively.  
We presented numerical results for denoising of artificial InSAR data as well as an example of inpainting. 
Moreover, several conjectures are at hand on relations between double metric integral regularization 
functionals and single integral representations.

\subsection*{Acknowledgements}
We thank Peter Elbau for very helpful discussions and comments. MH and OS acknowledge support from 
the Austrian Science Fund (FWF) within the national research network Geometry and Simulation, project S11704 
(Variational Methods for Imaging on Manifolds). Moreover, OS is supported by the Austrian Science Fund (FWF), 
with SFB F68, project F6807-N36 (Tomography with Uncertainties) and I3661-N27 (Novel Error Measures and Source 
Conditions of Regularization Methods for Inverse Problems).


\section*{References}
\renewcommand{\i}{\ii}
\printbibliography[heading=none]



\end{document}